\pdfoutput=1 % for submission to arxiv.org
\documentclass[reqno,11pt]{amsart}
\usepackage[utf8]{inputenc}
\usepackage[dvipsnames]{xcolor}
\usepackage{amssymb,amsmath,amsthm, amsbsy,
multicol, mathtools, 
microtype, % improves appearance
tikz,
enumerate
}

\allowdisplaybreaks % allow align environment to change page

\usepackage[margin=1.2in]{geometry}
\usepackage{ytableau}
\usepackage[vcentermath,enableskew]{youngtab}

\usepackage[colorlinks=true, pdfstartview=FitV, 
linkcolor=violet, 
citecolor=blue,
filecolor=violet, urlcolor=violet]{hyperref}

\newcommand{\arxiv}[1]{\href{http://arxiv.org/abs/#1}{\texttt{arXiv:#1}}}

\theoremstyle{definition} % make everything non-italics
\newtheorem{theorem}{Theorem}[section]
\newtheorem{lemma}[theorem]{Lemma}
\newtheorem{conjecture}[theorem]{Conjecture}
\newtheorem{proposition}[theorem]{Proposition}
\newtheorem{corollary}[theorem]{Corollary}

\newtheorem{definition}[theorem]{Definition}
\newtheorem{example}[theorem]{Example}
\newtheorem{remark}[theorem]{Remark}
\numberwithin{equation}{section}
\counterwithout{figure}{section}

\newtheorem{thmIntro}{Theorem}

\definecolor{forGreen}{RGB}{0,120,0}

\DeclareMathOperator{\sh}{sh}
\DeclareMathOperator{\SD}{SD}
\DeclareMathOperator{\BB}{BB}
\DeclareMathOperator{\ID}{ID}
\DeclareMathOperator{\RSK}{RS}

\DeclareMathOperator{\des}{des}
\DeclareMathOperator{\CA}{CA}
\DeclareMathOperator{\incr}{i}
\DeclareMathOperator{\decr}{d}
\DeclareMathOperator{\localincr}{I}
\DeclareMathOperator{\localdecr}{D}
\renewcommand{\P}{\text{P}}
\DeclareMathOperator{\Q}{Q}
\DeclareMathOperator{\len}{len}
\newcommand{\Rowone}{{\sf{{Row}}}_1}
\renewcommand{\omega}{w}
\DeclareMathOperator{\SST}{sst}



\title[BBS and RSK recording tableaux]{Box-ball systems and RSK recording tableaux}
\author[MC, OF, EG, MS, DZ]{
Marisa Cofie, Olivia Fugikawa, Emily Gunawan, Madelyn Stewart, and David Zeng}
\begin{document}

\begin{abstract}
A box-ball system (BBS) is a discrete dynamical system consisting of $n$ balls in an infinite strip of boxes. During each BBS move, the balls take turns jumping to the first empty box, beginning with the smallest-numbered ball. The one-line notation of a permutation can be used to define a BBS state. This paper proves that the Robinson--Schensted (RS) recording tableau of a permutation completely determines the dynamics of the box-ball system containing the permutation.

Every box-ball system eventually reaches steady state, decomposing into solitons. We prove that the rightmost soliton is equal to the first row of the RS insertion tableau and it is formed after at most one BBS move. This fact helps us compute the number of BBS moves required to form the rest of the solitons. First, we prove that if a permutation has an L-shaped soliton decomposition then it reaches steady state after at most one BBS move. Permutations with L-shaped soliton decompositions include noncrossing involutions and column reading words. Second, we  make partial progress on the conjecture that every permutation on $n$ objects reaches steady state after at most $n-3$ BBS moves. Furthermore, we study the permutations whose soliton decompositions coincide with standard tableaux; we conjecture that they are closed under consecutive pattern containment and that the RS recording tableaux belonging to such permutations are counted by the Motzkin numbers.
\end{abstract}

\maketitle

\section{Introduction}
\subsection{Box-ball systems}

\begin{figure}[htb!]
\begin{align*}
t=0 \text{\hspace{5mm}}
\begin{ytableau}
    \none[\cdots] & 
    &&
    \textcolor{red}{4} &
    \textcolor{forGreen}{5}& 
    \textcolor{orange}{2} &
    \textcolor{blue}{3}&
    \textcolor{magenta}{6} &
    \textcolor{purple}{1}&
    &
    &
    &&\none[\cdots]
    \end{ytableau}\\
    \begin{ytableau}
    \none[\cdots] & 
    &&
    \textcolor{red}{4} &
    \textcolor{forGreen}{5}& 
    \textcolor{orange}{2}&
    \textcolor{blue}{3}&
    \textcolor{magenta}{6} &
    &
    \textcolor{purple}{\bf 1}&
    &
    &&\none[\cdots]
    \end{ytableau}\\
    \begin{ytableau}
    \none[\cdots] & 
    &&
    \textcolor{red}{4} &
    \textcolor{forGreen}{5}&
    &
    \textcolor{blue}{3}&
    \textcolor{magenta}{6} &
    \textcolor{orange}{\bf 2}&
    \textcolor{purple}{1}&
    &
    &&\none[\cdots]
    \end{ytableau}\\
    \begin{ytableau}
    \none[\cdots] & 
    &&
    \textcolor{red}{4} &
    \textcolor{forGreen}{5}&
    &
    &
    \textcolor{magenta}{6} &
    \textcolor{orange}{2}&
    \textcolor{purple}{1}&
    \textcolor{blue}{\bf 3}&
    &&\none[\cdots]
    \end{ytableau}\\
    \begin{ytableau}
    \none[\cdots] & 
    &&
    \text{} &
    \textcolor{forGreen}{5}&
    \textcolor{red}{\bf 4} &
    &
    \textcolor{magenta}{6} &
    \textcolor{orange}{2}&
    \textcolor{purple}{1}&
    \textcolor{blue}{3}&
    &&\none[\cdots]
    \end{ytableau}\\
    \begin{ytableau}
    \none[\cdots] & 
    &&
    \text{} &
    \text{} &
    \textcolor{red}{4} &
    \textcolor{forGreen}{\bf 5}&
    \textcolor{magenta}{6} &
    \textcolor{orange}{2}&
    \textcolor{purple}{1}&
    \textcolor{blue}{3}&
    &&\none[\cdots]
    \end{ytableau}\\
    t=1 \text{\hspace{5mm}}
    \begin{ytableau}
    \none[\cdots] & 
    &&
     \text{} &
     \text{} &
    \textcolor{red}{4} &
    \textcolor{forGreen}{5}&
    &
    \textcolor{orange}{2}&
    \textcolor{purple}{1}&
    \textcolor{blue}{3}&
    \textcolor{magenta}{\bf 6} 
    &&\none[\cdots]
    \end{ytableau}
    \end{align*}
\caption{One BBS move from $\dots ee452361eeee\dots$ to  $ \dots eeee45e2136e \dots$}    
\label{fig:intro:452361:t=0 to t=1}  
\end{figure}

The \emph{box-ball system}, 
or BBS for short, 
is a dynamical system consisting of discrete time states. 
At each time state, we have 
finitely many 
numbered balls in an infinite strip of boxes;
the boxes are indexed by the integers from left to right, and 
each box can fit at most one ball.  
One BBS move is the process of letting each ball jump  to the nearest empty box to its right, starting with the smallest-numbered ball (see Figure~\ref{fig:intro:452361:t=0 to t=1}). 
Given a BBS state at time $t$, we compute the BBS state at time $t+1$ by applying one BBS move.

Let $S_n$ denote the set of permutations on $[n] \coloneqq \{ 1, 2, \dots, n\}$. 
A permutation $w$ in $S_n$ gives a box-ball system state by assigning the one-line notation of the permutation to 
$n$ 
consecutive
boxes. 
We denote an empty box by $e \coloneqq n+1$, 
and we usually omit the infinitely many empty boxes to the left of the balls (even though our boxes are indexed by $\mathbb{Z}$). 
Let $\BB^t(X)$ denote the result of applying $t$ BBS moves to a BBS configuration $X$. 
For example, 
beginning with a configuration 
\[\BB^0(X) = 452361eeeeeeeee\cdots\]
at time $t=0$, 
one BBS move (in which all balls jump once, starting with ball $1$ and ending with ball $6$) results in the new configuration (at $t=1$) 
\[\BB^1(X) = ee45e2136eeeeee\cdots.\] 
A second BBS move produces the $(t=2)$ configuration \[\BB^2(X) = eeee452ee136eee\cdots,\] and a third BBS move produces the $(t=3)$ configuration \[\BB^3(X) = eeeeee425eee136\cdots.\] 
At every subsequent time step, the three balls $136$ advance three spaces to the right, the pair $25$ advances two spaces to the right, and the singleton $4$  advances one space to the right. See Figure~\ref{fig:intro:452361:t=0 to t=4}. 
These blocks 
are called \emph{solitons} --- maximal consecutive increasing sequences 
of balls 
that are preserved by all future BBS moves. 
The configurations where $t \geq 3$ are said to be in \emph{steady state}, because each ball is contained in a soliton. 
The \emph{steady-state time} of this permutation (the number of BBS moves required to reach steady state) is $t=3$.

\begin{figure}[hbt]
\begin{align*}
t=0 \text{\hspace{5mm}}
\begin{ytableau}
\textcolor{red}{\bf 4} &
\textcolor{forGreen}{\bf 5}& 
\textcolor{orange}{\bf 2} &
\textcolor{blue}{\bf 3}&
\textcolor{magenta}{\bf 6} &
\textcolor{purple}{\bf 1}&
    &
    &
    &
    &
    &
    &
    &
    &
    &
    &
    &
    &
    &
    &
    & \none[\cdots] 
    \end{ytableau}\\
    t=1 \text{\hspace{5mm}}
    \begin{ytableau}
    \text{} &
    \text{} &
    \textcolor{red}{\bf 4} &
    \textcolor{forGreen}{\bf 5}& 
    \text{} &
    \textcolor{orange}{\bf 2} &
\textcolor{purple}{\bf 1}&
\textcolor{blue}{\bf 3}&
\textcolor{magenta}{\bf 6}&
    &
    &
    &
    &
    &
    &
    &
    &
    &
    &
    &
    & \none[\cdots] 
    \end{ytableau}\\
    t=2 \text{\hspace{5mm}}
    \begin{ytableau}
    \text{} &
    \text{} &
    \text{} &
    \text{} &
    \textcolor{red}{\bf 4} &
    \textcolor{forGreen}{\bf 5}& 
    \textcolor{orange}{\bf 2} &
    \text{} &
    \text{} &
\textcolor{purple}{\bf 1}&
\textcolor{blue}{\bf 3}&
\textcolor{magenta}{\bf 6}&
    &
    &
    &
    &
    &
    &
    &
    &
    & \none[\cdots] 
    \end{ytableau}\\
    t=3 \text{\hspace{5mm}}
    \begin{ytableau}
    \text{} &
    \text{} &
    \text{} &
    \text{} &
    \text{} &
    \text{} &
    \textcolor{red}{\bf 4} &
    \textcolor{orange}{\bf 2} &
    \textcolor{forGreen}{\bf 5}& 
    \text{} &
    \text{} &
    \text{} &
\textcolor{purple}{\bf 1}&
\textcolor{blue}{\bf 3}&
\textcolor{magenta}{\bf 6}&
    &
    &
    &
    &
    &
    & \none[\cdots] 
    \end{ytableau}\\
    t=4 \text{\hspace{5mm}}
    \begin{ytableau}
    \text{} &
    \text{} &
    \text{} &
    \text{} &
    \text{} &
    \text{} &
    \text{} &    
    \textcolor{red}{\bf 4} &
    \text{} &
    \textcolor{orange}{\bf 2} &
    \textcolor{forGreen}{\bf 5}& 
    \text{} &
    \text{} &
    \text{} &
    \text{} &
\textcolor{purple}{\bf 1}&
\textcolor{blue}{\bf 3}&
\textcolor{magenta}{\bf 6}&
    &
    &
    & \none[\cdots] 
    \end{ytableau}\\
    t=5 \text{\hspace{5mm}}
    \begin{ytableau}
    \text{} &
    \text{} &
    \text{} &
    \text{} &
    \text{} &
    \text{} &
    \text{} &    
    \text{} &    
    \textcolor{red}{\bf 4} &
    \text{} &    
    \text{} &    
    \textcolor{orange}{\bf 2} &
    \textcolor{forGreen}{\bf 5}& 
    \text{} &
    \text{} &    
    \text{} &    
    \text{} &
    \text{} &
\textcolor{purple}{\bf 1}&
\textcolor{blue}{\bf 3}&
\textcolor{magenta}{\bf 6} 
& \none[\cdots] 
    \end{ytableau}
\end{align*}
\caption{A box-ball system with the permutation $452361$  at $t=0$}
\label{fig:intro:452361:t=0 to t=4}
\end{figure}

Every box-ball system eventually reaches steady state,  %after a finite number of BBS moves, 
decomposing into solitons whose sizes are weakly increasing from left to right, 
i.e., forming an integer partition. %of $n$. 
We can encode this \emph{soliton decomposition} of the box-ball system in a  tableau whose first row is the rightmost soliton,
the second row is the second rightmost soliton, and so on. 
Note that each row of this tableau is necessarily an increasing sequence, but the columns do not have to be increasing. 
The shape of the soliton decomposition is called the \emph{BBS soliton partition}.

Given a permutation $\omega$, its soliton decomposition $\SD(\omega)$ is the soliton decomposition of the box-ball system containing $\omega$. 
For example, the soliton decomposition of the permutation $\omega = 452361$ is 
\[ \SD(\omega) = \begin{ytableau}
\textcolor{purple}{\bf 1}&
\textcolor{blue}{\bf 3}&
\textcolor{magenta}{\bf 6}\\
\textcolor{orange}{\bf 2} & \textcolor{forGreen}{\bf 5} \\
\textcolor{red}{\bf 4}
\end{ytableau}\]

Our version of the box-ball system is know as the \emph{multicolor box-ball system}. It was introduced by Takahashi in \cite{Takahashi93} and is a generalization of the box-ball system (with unlabeled balls) first invented by Takahashi and Satsuma in \cite{TS90}. The box-ball system arises in quantum integrable systems and classical integrable systems via procedures called crystallization and ultradiscretization, respectively. It has connections to crystal base theory in quantum groups, solvable lattice models, combinatorial Bethe ansatz, tropical geometry, and more; see the survey~\cite{IKT12}.

\subsection{Robinson--Schensted tableaux}
A tableau is called \emph{standard} if 
the entries in its rows and columns are increasing and each of the integers in $[n]$ %$1$ through $n$ 
appears exactly once. 
A popular way to associate standard tableaux to permutations is via the Robinson--Schensted (RS) correspondence 
\[{w \mapsto (\P(w), \Q(w))}\]
from $S_n$ onto pairs of standard size-$n$ tableaux of the same shape~\cite{Sch61}.
The tableau $\P(w)$ is called the \emph{insertion tableau} of  $w$, and the tableau $\Q(w)$ is called the \emph{recording tableau} of $w$. The shape of these tableaux is called the \emph{RS partition} of $w$. 
For more details, see for example the textbook~\cite[Chapter~3]{Sag01}.

Schensted's classical theorem 
 says that 
the size of the first row (respectively, first column) of the RS partition of $w$ is equal to the length of a longest increasing (respectively, decreasing) subsequence of the one-line notation of $w$. 
A localized version of Schensted's theorem 
due to Lewis, Lyu, Pylyavskyy, and Sen 
interprets the size of the first row and the size of the first column of the BBS soliton partition as certain preserved statistics 
in a box-ball system. 
We discuss both theorems in Section \ref{sec:Schensted's theorem and local Schensted's theorem}.

As noted earlier, 
the soliton decomposition 
$\SD(w)$ of a permutation $w$ is not necessarily a standard tableau.
However, 
it is shown in~\cite{DGGRS23} that 
$\SD(w)$ is a standard tableau 
iff 
its shape coincides with the RS partition of $w$
iff $\SD(w)=\P(w)$. 
This 
connection between the soliton decomposition of a permutation and its insertion tableau 
motivates us to define a permutation $w$ to be \emph{BBS good} (\emph{good} for short) if $\SD(w)$ is a standard tableau. 
We conjecture that good permutations are closed under consecutive pattern containment (Conjecture~\ref{conj:good permutations closed}).

\subsection{RS recording tableaux}

Having seen the relationship between BBS soliton decompositions and RS insertion tableaux described in the previous paragraph, 
it is natural to ask whether RS recording tableaux may play a role in the study of box-ball systems. 
Surprisingly, the recording tableau of a permutation completely determines the BBS dynamics of the permutation, in the following sense.

\begin{thmIntro}
\label{thmIntro:Q determines steady state and Q determines BBS partition}
If $\pi, \omega$ are permutations such that $\Q(\pi) = \Q(\omega)$, 
then the following holds.
\begin{enumerate}
\item $\pi$ and $\omega$ have the same steady-state time (Theorem~\ref{thm:Q determines steady state})
\item 
The shape of $\SD(\pi)$ equals the shape of $\SD(\omega)$
(Theorem~\ref{thm:Q determines BBS partition})
\item $\pi$ is good iff $\omega$ is good (Theorem~\ref{thm:Q is either good or bad})
\end{enumerate}
\end{thmIntro}

The last theorem tells us that the recording tableau $\Q(w)$ determines whether or not $w$ is good, 
so
we define a standard tableau $T$ to be \emph{good} 
if $\Q(\omega) = T$ implies $\omega$ is good, 
equivalently,  
if $\Q(\omega) = T$ for some good permutation $\omega$.
We conjecture that good tableaux are counted by the Motzkin numbers (Conjecture~\ref{conj:Motzkin}).

\subsection{First solitons and steady-state times}

In Section~\ref{sec:first soliton}
(respectively, \ref{sec:L-shaped} and \ref{sec:maximum sst}),
 we study the number of BBS moves required to create 
the rightmost soliton (respectively, all solitons).

In Section \ref{sec:first soliton}, we prove that applying one BBS move to a permutation is enough to produce the rightmost soliton of the box-ball system.
\begin{thmIntro}[Theorem~\ref{thm:first soliton created after 1 BBM}]
\label{thmIntro:first soliton created after 1 BBM}
If $X$ is a BBS configuration corresponding to a permutation $w$,
then the 
rightmost soliton is created after applying at most one BBS move to $X$, and 
this rightmost soliton 
is equal to the first row of $\P(w)$.  
\end{thmIntro}

Theorem~\ref{thmIntro:first soliton created after 1 BBM} is helpful for proving the rest of our results. 
% The \emph{steady-state time} of a permutation $w$ is the number of BBS moves required for $w$ to reach steady state.

\begin{thmIntro}
[Theorem~\ref{thm: L-shaped SD SST}]\label{thmIntro: L-shaped SD SST}
If a permutation $w$ has an L-shaped soliton decomposition, that is, the shape of $\SD(w)$ is of the form $(s,1,1,\dots)$, then the steady-state time of $w$ is either $0$ or~$1$.
\end{thmIntro}

% In Section~\ref{sec:L-shaped}, 
% we also show that permutations whose soliton decompositions are L-shaped include column reading words and noncrossing involutions.
In particular, we show in Section~\ref{sec:L-shaped} that such permutations include column reading words and noncrossing involutions.

We also investigate upper bounds of steady-state times. 
It was conjectured in~\cite[Conjecture~1.1]{DGGRS23} that the steady-state time of a permutation in $S_n$ is at most $n-3$. 
We prove a special case of this conjecture. 
\begin{thmIntro}[Theorem~\ref{thm:max SST of T with sh(n-3,2,1)}]
All permutations 
with RS partition 
$(n-3,2,1)$ 
have steady-state time at most $n-3$.
\end{thmIntro}

In Section~\ref{sec:chain of tableaux}, we use Bender--Knuth involutions to construct 
a sequence of 
tableaux 
with steady-state times from 0 to $n-3$.

%%%
%%%

\section{Steady State}
\label{sec:steady state}

A BBS configuration $X$ is a sequence indexed by $\mathbb{Z}$ where each 
number in $[n]$ 
(each denoting a ball) appears exactly once, and $e \coloneqq n+1$ (denoting an empty box) 
appears infinitely many times. 
Let $\BB^t(X)$ denote the result of applying $t$ BBS moves to a BBS configuration $X$. 
% A \emph{soliton} is a maximal consecutive increasing sequence of balls that is preserved by all subsequent BBS moves. 
% A BBS configuration is said to be in \emph{steady state}
% if every ball is contained in a soliton.

% \begin{definition}
% \label{def:runs}
An \emph{increasing run} of a permutation
is a maximal increasing contiguous nonempty subsequence.
For example, the increasing runs of the permutation 452361 are 45, 236, and 1. 
An \emph{increasing run} of a BBS configuration is a maximal increasing contiguous nonempty sequence of balls. 
% \end{definition}

A \emph{soliton} is an increasing run %a maximal consecutive increasing sequence of balls 
that is preserved by all subsequent BBS moves. 
A BBS configuration is said to be in \emph{steady state}
if every ball is contained in a soliton. 
The \emph{steady-state time} of a permutation $w$ is the number of BBS moves required for $w$ to reach steady state.

\subsection{Increasing decomposition and steady state}
In this section, we give a set of criteria for 
steady state. 
\begin{definition}
The \emph{increasing run decomposition}
of a BBS configuration $X$, denoted by $\ID(X)$, 
is the table where the rightmost increasing run of $X$ is the first (top) row, the next increasing run to its left is the second row, and so on.
\end{definition}
\begin{example}
Let
$%\BB^2(X)
X=
e \, e \, e \, e  \, 4 \, 5 \, 2 \, e \, e \, 1 \, 3 \, 6 \, e \, e \, e\cdots$. 
Then 
\[
\ID(X) = 
{\young(136,2,45)}
\]
\end{example}
\begin{remark}
\label{rem:height of ID is an invariant}
A special case of {\cite[Lemma 3.5]{LLPS24}} is that the height of the increasing run decomposition (i.e. the number of increasing runs) is an invariant of the box-ball system, that is, 
the number of rows in $\ID(X)$ is equal to that of $\ID(\BB^t(X))$ for all $t \in \mathbb{Z}$.
See Theorem~\ref{thm:local Schensted's theorem for BBS}. 
\end{remark}
\begin{remark} A BBS configuration $X$ is in steady state iff
\[
\ID(\BB^t(X))=\ID(X)
\; \; \text{for each } 
t=1,2,3,\dots
\]
If $X$ is a BBS configuration which is in steady state,  
then by definition $\ID(X)$ is equal to the soliton decomposition of the BBS system.
\end{remark}
The following gives a way to check whether a BBS configuration has reached steady state. 
\begin{proposition}[Steady-state characterization using ID]
\label{prop:characterization of steady state}
A BBS configuration $X$ is in steady state iff
\begin{enumerate}
    \item  \label{lem:characterization of steady state:1}
    the rows of $\ID (X)$ are weakly decreasing in length, and 
    \item \label{lem:characterization of steady state:2}
    $\ID(\BB(X)) = \ID(X)$
\end{enumerate}
\end{proposition}

\subsection{Configuration array and steady state}
\label{sec:configuration array}
A BBS state $X$ can be represented by the \emph{configuration array} $\CA(X)$ containing the integers from 1 to $n$ as follows: scanning the boxes from right to left, each increasing run becomes a row in the array. 
A string of $g$ empty boxes 
 indicates that the next row below should be shifted $g$ spaces to the left.
Note that this array has increasing rows but not necessarily increasing columns; it may be disconnected and it may not have a valid skew shape. 

The following is a corollary of a characterization for steady state (called `separation condition') given in \cite{LLPS24}. For a proof, see {\cite[Section~5]{DGGRS23}}.

\begin{proposition}[Steady-state characterization using CA]\label{prop:t=0 generalization} 
A BBS configuration $X$ is in steady state 
iff 
its configuration array $\CA(X)$
is a standard (possibly disconnected) skew tableau whose 
 rows are weakly decreasing in length.
\end{proposition}

\begin{example}
Let $w= 452361$, the example from Figure \ref{fig:intro:452361:t=0 to t=4}. 
The following are the box-ball system states from time $t=0$ to $4$ and their configuration arrays.

\begin{align*}
% t=0 \qquad 
&
\BB^0(X) = 4 \, 5 \, 2 \, 3 \, 6 \, 1 \, e \, e \,e \, e\, e \, e \, e \, e \, e\dots
& {\young(1,236,45)} 
\\[3mm]
% t=1 \qquad 
& 
\BB^1(X) = e \, e \, 4 \, 5 \, e \, 2 \, 1 \, 3 \, 6 \, e \, e \, e \, e \, e \, e\cdots
& 
{\young(:136,:2,45)}
\\[3mm]
% t=2 \qquad 
& 
\BB^2(X) = e \, e \, e \, e  \, 4 \, 5 \, 2 \, e \, e \, 1 \, 3 \, 6 \, e \, e \, e\cdots
&
{\young(::136,2,45)}
\\[3mm]
% t=3 \qquad 
& 
\BB^3(X) = e \, e \, e \, e \, e \, e \, 4 \, 2 \, 5 \, e \, e \, e \, 1 \, 3 \, 6\cdots
&
{\young(:::136,25,4)}
\\[3mm]
% t=4 \qquad 
& 
\BB^4(X) = e \, e \, e \, e \, e \, e \, e \, 4 \, e \, 2 \, 5 \, e \, e \, e \, e \, 1 \, 3 \, 6\cdots
&
{\young(:::::136,:25,4)}
\end{align*}
In this box-ball system, all configurations at time $t \geq 3$ are in steady state, so the steady-state time of $452361$ is $3$.
\end{example}

\begin{remark}
\label{rem:t=0}
The \emph{row reading word} of a % (not-necessarily standard)
 tableau is the permutation formed by concatenating the rows of the tableau from bottom to top. 
For instance, $425136$ is the row reading word of the standard tableau
\[\young(136,25,4)
\]
It follows from Proposition~\ref{prop:t=0 generalization} 
 that  
a permutation has steady-state time $0$ iff 
it is the row reading word of a standard tableau. 
\end{remark}

%%%%%%%%%
%%%%%%%%%%

\section{A localized version of Schensted's theorem for box-ball systems}
\label{sec:Schensted's theorem and local Schensted's theorem}

% The well-known 
Schensted's theorem explores connection between the RS partition and lengths of increasing and decreasing subsequences. 
A localized version of Schensted's theorem by Lewis, Lyu, Pylyavskyy, and Sen explores similar connection between the BBS soliton partition and certain invariants of the BBS system. 
We describe both theorems in this section.

\subsection{Schensted's theorem and RS partition}

In {\cite[Theorem~1]{Sch61}}, Schensted gives 
meaning 
to the first row and the first column of an RS partition.

Let $\incr(w)$ (respectively, $\decr(w)$) denote the size of a longest increasing (respectively, decreasing) subsequence of the one-line notation of a permutation $w$.

\begin{theorem}
[Schensted's theorem]
\label{thm:Schensted's theorem}
The size of the first row (respectively, first column) of the RS partition of a permutation $w$ is equal to $\incr(w)$ 
(respectively, $\decr(w)$).
\end{theorem}

\begin{example}
\label{ex:thm:Greene's theorem:5623714}
Let $w=5623714$.  
The longest increasing subsequences are 
$
5 6 7$, 
$2 3 7$,
and 
$2 3 4$, % checked with SageMath
so $\incr(w) =3$. 
The longest %$1$-
decreasing subsequences are 
$521$, 
$621$, 
$531$, and
$631$,  % checked with SageMath
so $\decr(w) =3$. 
The corresponding RS tableaux are
\[
\P(w)=\young(134,267,5), 
\qquad
\Q(w)=
\young(125,347,6)
% Computed with SageMath
\]
\end{example}

Schensted's theorem has an important generalization in Greene's theorem (\cite[Theorem ~3.1]{Gre74}), 
which interprets the RS partition as sizes of largest unions of increasing and decreasing sequences.
For more details, see for example Chapter 3 of the textbook~\cite{Sag01}.

\subsection{Localized Schensted's theorem and BBS soliton partition}
\label{sec:local Schensted's theorem}

In~\cite[Lemma 3.5]{LLPS24},  
Lewis, Lyu, Pylyavskyy, and Sen present  
a localized version of 
Greene's theorem for box-ball systems.  
In this section we discuss a special case of their result, reframed to match our box-ball convention.
We are calling this special case ``a localized version of Schensted's theorem".
The reason is that,  
when adapted to permutations,   
the statement of their result 
can be obtained from the original 
Schensted's theorem with RS partition replaced by BBS soliton partition and with  ``size of a longest decreasing subsequence" replaced by ``number of descents plus 1".

Given a (possibly infinite) sequence $X$, 
an integer $j$ is called a \emph{descent} of $X$ if $X(j) > X(j+1)$.

\begin{theorem}[Localized Schensted's theorem for permutations]
\label{thm:local Schensted's theorem}
Suppose $w$ is a permutation.
\begin{enumerate}
\item 
The size of the first row of the BBS soliton partition of 
$w$ is equal to $\incr(w)$, the size 
of a longest increasing  subsequence of $w$.
\item 
The size of the first column of the BBS soliton partition of 
$w$ is equal to $1 + |\{ \text{descents of } w \}|$, denoted by $\localdecr(w)$.
\end{enumerate}
\end{theorem}

\begin{example}
\label{ex:lem:local Greene's theorem:5623714}
Let $w=5623714$, the permutation from Example~\ref{ex:thm:Greene's theorem:5623714}. 
Then $\incr(w) = 3 $
as computed earlier, 
and
$\localdecr (w) = 1 + |\{\text{descents of } 5623714\}| = 1 + |\{ 2, 5 \}| = 3.$  
Note that 
the soliton decomposition 
\[
\SD(w)=
\young(134,27,56)
\]
is a nonstandard tableau because the second column of the tableau is $376$, a non-increasing sequence, when read from top to bottom. We see that
$\sh \SD(w)=(3,2,2)$ while $\sh \P(w)=(3,3,1)$. 
In particular, 
$\SD(w) \neq \P(w)$, 
see
Theorem~\ref{thm:tfae}. 
\end{example}

\begin{remark}
\label{rem:sizes of RS and BBS}
Let $w$ be a permutation.
\begin{enumerate}
\item 
\label{itm:rem:lambdaRS1 equals lambdaBBS1}
Schensted's theorem (Theorem~\ref{thm:Schensted's theorem}) and localized Schensted's theorem for permutations (Theorem~\ref{thm:local Schensted's theorem}) tells us that 
\begin{align*}
\text{(the size of the first row of $\P(w)$)}
&=
\text{(the size of the first row of $\SD(w)$)}\\
 &=
\text{(the size of the first (rightmost) soliton of $w$)}.
\end{align*}

\item
\label{itm:rem:perm stats inequalities}
It follows from the definition that 
\[
  \decr(\omega)
 \leq 
 \localdecr(\omega).
\]
Combining this with 
Theorem~\ref{thm:Schensted's theorem} 
and
Theorem~\ref{thm:local Schensted's theorem}, 
\begin{align*}
\text{
(the size of the first column of $\P(w)$)}
&\leq 
\text{(the size of the first column of $\SD(w)$)}\\
&=
\text{(the number of solitons of $w$).}
\end{align*}
\end{enumerate}
\end{remark}

The statistics in Theorem~\ref{thm:local Schensted's theorem}
can be defined for general BBS configurations.
Recall that a BBS configuration
is a sequence indexed by $\mathbb{Z}$ where each number in $[n]$ %$1,\dots,n$ 
(denoting balls) appears exactly once, and $e \coloneqq n+1$ (denoting empty boxes) appears infinitely many times.

\begin{definition}
Let $X$ be a BBS configuration with $n$ distinct balls labeled by $[n]$.
\begin{enumerate} 
\item 
Given a finite, increasing sequence $u$ of balls in $X$, the \emph{penalized length} of $u$ in $X$ is the number of balls in $u$ minus the number of gaps (i.e., empty boxes) between the first and last balls of $u$.
Let $\localincr(X)$ denote the maximum penalized length of increasing subsequences of balls in $X$. 

\item 
Let $\localdecr(X)$ denote the number of descents of $X$, equivalently, the number of rows of $\ID(X)$.

\end{enumerate}
\end{definition}

Note that, since $X$ consists of $n$ balls and the empty boxes have values $e=n+1$,  
we have
$1 \leq \localdecr(X), ~\localincr(X) \leq n$. 
If the leftmost ball is in box $j$, then $j-1$ is a descent of $X$, since $X(j-1)=e>X(j)$.

\begin{example}
Consider the following BBS configuration $X$:
\begin{center}
\begin{tabular}
{c||p{.2in}|p{.1in}|p{.1in}|p{.1in}|p{.1in}|p{.1in}|p{.1in}|p{.1in}|p{.1in}|p{.1in}|p{.13in}|p{.3in}}
\raisebox{0in}[.2in][.1in]{}$j$ & \dots & 1 & 2 & 3 & 4 & 5 & 6 & 7 & 8 & 9 & 10 & \dots \\
\hline
\raisebox{0in}[.2in][.1in]{}$X(j)$ & \dots & e & e & 4 & 5 & e & 2 & 1 & 3 & 6 & e &  \dots
\end{tabular}
\end{center}
Since the leftmost ball is in box $3$, the integer $2$ is a descent of $X$. 
The other descents of $X$ are $5$ and $6$, so $\localdecr(X)=3$. 
The number of descents of $X$ is equal to 
the number of rows of 
\[\ID(X)={\young(136,2,45)}\]
The penalized length of the length-$3$ increasing sequence of balls $4 \, 5 \, 6$ is $3-1=2$ because there is one empty box between $4$ and $6$. We have $\localincr(X)=3$ because the penalized length of $1 \, 3 \, 6$ is $3$.
\end{example}

\begin{theorem}[Localized Schensted's theorem for BBS configurations]
\label{thm:local Schensted's theorem for BBS}
Suppose $X$ is a BBS configuration with $n$ distinct balls labeled by $[n]$.
\begin{enumerate}
\item 
The statistic $\localincr(X)$ 
is an invariant of the box-ball system, 
that is, 
$\localincr(X)=\localincr(\BB(X))$.

\item 
The statistic $\localdecr(X)$ is 
 also preserved by BBS moves. 
That is,
$\localdecr(X)=\localdecr(\BB(X))$; in other words, the number of rows of $\ID(\BB(X))$  is equal to that of $\ID(X)$. 
\end{enumerate}

In particular, the size of the first row (respectively, column) of the soliton decomposition of the box-ball system is equal to $\localincr(X)$ (respectively,  $\localdecr(X)$). 
If $w$ is a permutation in $S_n$ such that $X=\BB^t(w)$ for some $t$, then $\localincr(X)=\incr(w)$ and $\localdecr(X)=\localdecr(w)$.
\end{theorem}

%%%%
%%%%
%%%%
\section{Recording tableau determines BBS dynamics}

The recording tableau completely determines the BBS dynamics of a permutation, in the following sense (Proposition~\ref{lem:Q determines dynamics of BBS}): if two permutations have the same 
$\Q$-tableau, 
all 
BBS configurations in the two corresponding box-ball systems are identical if we remove the ball labels. 
We prove 
that the recording tableau determines the steady-state time (Theorem~\ref{thm:Q determines steady state})  
and the BBS soliton partition (Theorem~\ref{thm:Q determines BBS partition}) of a permutation. 
We also define the notion of a \emph{BBS good} permutation and prove that the 
$\Q$-tableau 
determines whether or not a permutation is good (Theorem~\ref{thm:Q is either good or bad}).

\subsection{Dual Knuth equivalence}
 \label{sec:dual knuth}
We review the concept of dual Knuth equivalence, which was introduced by Haiman~\cite{Hai92}.

Two permutations $\pi, \omega \in S_n$ \emph{differ by a dual Knuth relation of the first kind} (denoted $\pi \overset{K_1^*}{\sim} \omega$), if for some $k$, 
\begin{align*}
\pi &= \ldots \textcolor{red}{\bf k+1} \ldots \textcolor{blue}{\bf k} \ldots \textcolor{forGreen}{\bf k+2}\ldots
\text{ and }
\\
\omega &= \ldots \textcolor{forGreen}{\bf k+2} \ldots \textcolor{blue}{\bf k} \ldots \textcolor{red}{\bf k+1} \ldots .
\end{align*}
\text{ or vice versa.}
They \emph{differ by a dual Knuth relation of the second kind} (denoted $\pi \overset{K_2^*}{\sim} \omega$), if for some $k$,   
\begin{align*}
\pi =  \ldots \textcolor{red}{\bf k} \ldots &\textcolor{purple}{\bf k+2} \ldots \textcolor{forGreen}{\bf k +1}\ldots  
\text{ and }
\\ 
\omega = \ldots \textcolor{forGreen}{\bf k+1} \ldots &\textcolor{purple}{\bf k+2} \ldots \textcolor{red}{\bf k}\ldots \end{align*}
\text{ or vice versa.}

The two permutations are dual Knuth equivalent  
if there is a sequence of permutations such that 
$$\pi =\pi_1  \overset{K_{i_1}^*}{\sim} \pi_2 \overset{K_{i_2}^*}{\sim}
\cdots \overset{K_{i_{k-1}}^*}{\sim} \pi_k=\omega$$

\noindent where $i_1, i_2, \dotsc, i_{k-1} \in \{1,2\}$.

Two permutations $\pi, \omega \in S_n$ are said to be \textit{$\Q$-equivalent} 
if $\Q(\pi) = \Q(\omega)$. 
A very helpful fact from  {\cite[Proposition 2.4]{Hai92}} tells us that 
the dual Knuth equivalence classes and $\Q$-equivalence classes coincide. 

\begin{theorem}
$
\Q(\pi) = \Q(\omega)$
 iff $\pi$ and $\omega$ are dual Knuth equivalent.
\end{theorem}

\subsection{BBS moves preserve dual Knuth equivalence}
Recall that a BBS configuration
is a sequence indexed by $\mathbb{Z}$ where each of $1,\dots,n$ 
(denoting balls) appears exactly once, and $e \coloneqq n+1$ (denoting empty boxes) appears infinitely many times. 
We extend the definition 
of dual Knuth relation to BBS configurations but insist that the two entries being swapped must be balls.  
Let $\BB(X)$ denote the result of applying one BBS move to a BBS configuration $X$.

The following lemma is key to proving the results of this section.

\begin{lemma}
\label{lem:bbs preserves knuth}
Let $X$ and $Y$ be two BBS configurations.  

\begin{enumerate}
\item 
Suppose $X$ and $Y$ differ by a dual Knuth relation of the first kind, 
say, 
\[ X= \dots \textcolor{red}{\bf k+1} \dots \textcolor{blue}{\bf k} \dots \textcolor{forGreen}{\bf k+2} \dots, ~ Y= \dots \textcolor{forGreen}{\bf k+2} \dots \textcolor{blue}{\bf k} \dots \textcolor{red}{\bf k+1} \dots,  \]
where $\textcolor{forGreen}{\bf k+2}$ represents a ball (as opposed to an empty box $e=n+1$).
Then $\BB(X)$ and $\BB(Y)$ also differ by a dual Knuth relation of the first kind such that the relative order of the balls $\textcolor{blue}{\bf k}$, $\textcolor{red}{\bf k+1}$, and $\textcolor{forGreen}{\bf k+2}$ are preserved:
\[ \BB(X)= \dots \textcolor{red}{\bf k+1} \dots \textcolor{blue}{\bf k} \dots \textcolor{forGreen}{\bf k+2} \dots, ~ \BB(Y)= \dots \textcolor{forGreen}{\bf k+2} \dots \textcolor{blue}{\bf k} \dots \textcolor{red}{\bf k+1} \dots.  \]

\item 
Suppose $X$ and $Y$ differ by a dual Knuth relation of the second kind, 
say,
\[ X= \dots \textcolor{red}{\bf k} \dots \textcolor{purple}{\bf k+2} \dots \textcolor{forGreen}{\bf k+1} \dots, 
~ 
Y= \dots \textcolor{forGreen}{\bf k+1} \dots \textcolor{purple}{\bf k+2} \dots \textcolor{red}{\bf k} \dots,  \]
where $\textcolor{purple}{\bf k+2}$ represents a ball (as opposed to an empty box).
Then $\BB(X)$ and $\BB(Y)$ differ by a dual Knuth relation of \emph{some} kind, 
and $\BB(X)$ and $\BB(Y)$ differ by swapping either $\textcolor{red}{\bf k},\textcolor{forGreen}{\bf k+1}$ 
or $\textcolor{forGreen}{\bf k+1}, \textcolor{purple}{\bf k+2}$.
During the BBS move, consider the situation immediately after the balls
$1, 2, \dots, k-1$ have finished jumping.
 
\begin{enumerate}
\item 
If, immediately after all balls smaller than $\textcolor{red}{\bf k}$ have jumped, there is at least one empty box between $\textcolor{red}{\bf k}$ and $\textcolor{forGreen}{\bf k+1}$, then 
\[ \BB(X)= \dots \textcolor{red}{\bf k} \dots \textcolor{purple}{\bf k+2} \dots \textcolor{forGreen}{\bf k+1} \dots,
~ 
\BB(Y)= \dots \textcolor{forGreen}{\bf k+1} \dots \textcolor{purple}{\bf k+2} \dots \textcolor{red}{\bf k} \dots,  \]
that is, 
$\BB(X)$ and $\BB(Y)$ also differ by a dual Knuth relation of the second kind, and the relative order of the balls $\textcolor{red}{\bf k}$, $\textcolor{forGreen}{\bf k+1}$, and $\textcolor{purple}{\bf k+2}$ are the same as in $X$ and $Y$, respectively.

\item \label{lem:bbs preserves knuth:itm:second kind:case no empty box}
If, immediately after all balls smaller than $\textcolor{red}{\bf k}$ have jumped, there are no empty boxes between $\textcolor{red}{\bf k}$ and $\textcolor{forGreen}{\bf k+1}$,  
then 
\[ \BB(X)= \dots \textcolor{purple}{\bf k+2} \dots \textcolor{red}{\bf k} \dots \textcolor{forGreen}{\bf k+1} \dots,
~ 
\BB(Y)= \dots \textcolor{forGreen}{\bf k+1} \dots \textcolor{red}{\bf k} \dots \textcolor{purple}{\bf k+2} \dots,  \]
that is, $\BB(X)$ and $\BB(Y)$ differ by a dual Knuth relation of the \emph{first} kind.
\end{enumerate}
\end{enumerate}
In both cases, 
the positions of the nonempty boxes in $\BB(X)$ and $\BB(Y)$ are the same.
\end{lemma}

\begin{example}
% Example for case (2b)
To illustrate case~\eqref{lem:bbs preserves knuth:itm:second kind:case no empty box} of Lemma~\ref{lem:bbs preserves knuth}, 
consider 
the BBS configurations
\[
X=45\mathbf{\textcolor{red}{1}}\mathbf{\textcolor{purple}{3}}6\mathbf{\textcolor{forGreen}{2}} 
\text{ and } Y=45\mathbf{\textcolor{forGreen}{2}}\mathbf{\textcolor{purple}{3}}6\mathbf{\textcolor{red}{1}}.
\]
They differ by a dual Knuth relation of the \emph{second} 
kind where $k=1$, and  
there is no empty box between $k=1$ and $k+2=3$. 
We have
\begin{align*}
    & \BB(X)=45e\mathbf{\textcolor{purple}{3}}\mathbf{\textcolor{red}{1}}\mathbf{\textcolor{forGreen}{2}}6 \text{ and } %,\;\;\;\;\;\;\; 
    \BB(Y)=45e\mathbf{\textcolor{forGreen}{2}}\mathbf{\textcolor{red}{1}}\mathbf{\textcolor{purple}{3}}6.
\end{align*}
Indeed, $\BB(X)$ and $\BB(Y)$ differ by a dual Knuth relation of the \textit{first} kind.
\end{example}

\begin{lemma}
\label{lem:X is in steady state iff Y is in steady state}
\label{lem:dual knuth relation implies identical configuration array}
Suppose $X$ and $Y$ are two BBS configurations that differ by a dual Knuth relation;  
the two configurations are identical except that two balls $j$ and $j+1$ are swapped. 
\begin{enumerate}
\item \label{itm:lem:X is in steady state iff Y is in steady state}
Then $X$ is in steady state iff $Y$ is in steady state.
\item \label{itm:lem:dual knuth relation implies identical configuration array}
We have $\sh\ID(X) = \sh\ID(Y)$, and the gaps between the increasing runs are the same.
Equivalently, the configuration arrays $\CA(X)$ and $\CA(Y)$ have the same shape.
\end{enumerate}
\end{lemma} 
\begin{proof}
Since $X$ and $Y$ differ by a dual Knuth relation, 
the ball $j-1$ or the ball $j+2$ (possibly both) is located between $j$ and $j+1$ in both $X$ and $Y$.
This tells us that $j$ and $j+1$ are not in the same row of $\CA(X)$, since each row of $\CA(X)$ is an increasing run of $w$.

\eqref{itm:lem:X is in steady state iff Y is in steady state}
Suppose $X$ is in steady state. 
Then its configuration array $\CA(X)$ is a standard skew tableau whose rows are weakly decreasing in length by
Proposition~\ref{prop:t=0 generalization}. This requirement, combined with the fact that the ball $j-1$ or the ball $j+2$ is located between $j$ and $j+1$ in $X$, 
tells us $j$ and $j+1$ cannot be in the same column of $\CA(X)$. 
Since $j$ and $j+1$ are also not in the same row of $\CA(X)$,  swapping $j$ and $j+1$ in $\CA(X)$ would result in another standard skew tableau of the same shape. 
This standard skew tableau is $\CA(Y)$, so $Y$ is in steady state by Proposition~\ref{prop:t=0 generalization}.

\eqref{itm:lem:dual knuth relation implies identical configuration array}
Since $j$ and $j+1$ are not in the same row of $\CA(X)$, 
swapping $j$ and $j+1$ in $\CA(X)$ would result in another configuration array with the same shape as $\CA(X)$.
This new configuration array is $\CA(Y)$.
So $\sh\ID(X) = \sh\ID(Y)$, and the gaps between the increasing runs are the same.
\end{proof}

\subsection{Q determines the steady-state time of a permutation}
We now prove 
% Theorem \ref{thm:Q determines steady state}, which states 
that the recording tableau determines the steady-state time of a permutation.
Let $\BB^t(X)$ denote the result of applying $t$ BBS moves to a BBS configuration $X$.

\begin{theorem}
\label{thm:Q determines steady state}
If $\Q(\pi) = \Q(\omega)$ then $\pi$ and $\omega$ have the same steady-state time.
\end{theorem} 
\begin{proof} 
Assume $\Q(\pi) = \Q(\omega)$. Then $\pi$ and $\omega$ are related by a sequence of dual Knuth relations corresponding to swapping two balls $l$ times. 
Fix $t$, and let $X=\BB^t(\pi)$ and $Y=\BB^t(w)$. 
By Lemma ~\ref{lem:bbs preserves knuth}, applied $t$ times, we know that $X$ and $Y$ are also related by a sequence of $l$ dual Knuth relations.  
Therefore, Lemma~\ref{lem:X is in steady state iff Y is in steady state}\eqref{itm:lem:X is in steady state iff Y is in steady state}, applied $l$ times, tells us that that $X$ is in steady state iff $Y$ is in steady state. 
So $\pi$ and $\omega$ first reach steady state at the same time.
\end{proof}

\subsection{Q determines the BBS soliton partition of a permutation}

In this section, 
we prove 
Theorem ~\ref{thm:Q determines BBS partition}, which says that the recording tableau determines the 
BBS soliton partition of a permutation.

\begin{proposition}
\label{lem:Q determines dynamics of BBS} 
If $\Q(\pi)=\Q(\omega)$, 
then, 
at every $t$, 
\begin{enumerate}
\item 
the positions of the nonempty boxes in $\BB^t(\pi)$ and $\BB^t(\omega)$ are equal;
\item 
$\sh\ID(\BB^t(\pi)) = \sh\ID(\BB^t(\omega))$, and the gaps between the increasing runs are the same. 
Equivalently, $\CA(\BB^t(\pi))$ and $\CA(\BB^t(\omega))$ are of identical shape for each $t$.
\end{enumerate}
\end{proposition}
\begin{proof} 
Assume $\Q(\pi) = \Q(\omega)$. Then $\pi$ and $\omega$ are related by a sequence of dual Knuth relations corresponding to a sequence of $l$ two-ball swaps. 
We fix $t$, and 
let $X=\BB^t(\pi)$ and $Y=\BB^t(\omega)$.
By Lemma ~\ref{lem:bbs preserves knuth}, applied $t$ times, we know that 
$X$ and $Y$ 
are also related by a sequence of $l$ dual Knuth relations, and the nonempty boxes of $X$ and $Y$ are in the same positions.  
Therefore, Lemma~\ref{lem:dual knuth relation implies identical configuration array}\eqref{itm:lem:dual knuth relation implies identical configuration array}, applied $l$ times, tells us that the configuration arrays 
$\CA(X)$
and
$\CA(Y)$
are of identical shape.
\end{proof}

\begin{theorem}
\label{thm:Q determines BBS partition}
If $\Q(\pi) = \Q(\omega)$ then $\sh \SD(\pi) = \sh \SD(\omega)$.
\end{theorem}
\begin{proof} 
Suppose $\Q(\pi)=\Q(\omega)$. 
Let $t$ be such that $\BB^t(\pi)$ and  $\BB^t(\omega)$ are both in steady state.
Proposition~\ref{lem:Q determines dynamics of BBS} tells us that  
$\sh\ID(\BB^t(\pi))=\sh\ID(\BB^t(\omega))$.
Since $\BB^t(\pi)$ and $\BB^t(\omega)$ are in steady state, 
we have $\ID(\BB^t(\pi))=\SD(\pi)$ and  $\ID(\BB^t(\omega))=\SD(\omega)$.
Hence $\sh \SD(\pi) = \sh \SD(\omega)$. 
\end{proof}

\subsection{Good recording tableaux}

In general the soliton decomposition and the RS insertion tableau of a permutation do not coincide. 
However, 
the following shows that 
having a standard soliton decomposition tableau or having a BBS soliton partition which equals the RS partition is enough to guarantee that the soliton decomposition and the RS insertion tableau coincide.

\begin{theorem}[{\cite[Theorem 4.2]{DGGRS23}}]
\label{thm:tfae}
If $w$ is a permutation, then 
the following are equivalent:
\begin{enumerate}
\item 
$\SD(w) = \P(w)$

\item 
$\SD(w)$ is a standard tableau

\item 
The shape of $\SD(w)$ equals the shape of $\P(w)$
\end{enumerate}
\end{theorem}

We say that a permutation $w$ is 
\emph{BBS good}, or \emph{good} for short, if $\SD(w)$ is a standard tableau.
It turns out that $\Q(w)$ determines whether or not $w$ is good.

\begin{theorem}
\label{thm:Q is either good or bad}
Given a $\Q$-equivalence class, either all permutations in it are good or all of them are not good.
\end{theorem}

Therefore, it makes sense to define a standard tableau $T$ to be \emph{BBS good} 
if 
$\Q(\omega) = T$ implies
$\omega$ is good
(equivalently, if $\Q(\omega) = T$ for some $\omega$ which is good).

\begin{proof}[Proof of Theorem~\ref{thm:Q is either good or bad}]
Let $\Q(\pi)=\Q(\omega)$. 
Assume $\pi$ is good, that is, $\SD(\pi)$ is standard.  
Then 
\begin{align*}
\sh \SD(\omega) &= \sh \SD(\pi) \text{ by Theorem~\ref{thm:Q determines BBS partition}}
\\
 &= \sh \P(\pi) \text{ by Theorem~\ref{thm:tfae}} \\
 &= \sh \P(\omega),
\end{align*}
where the last equality is due to $\Q(\pi)=\Q(\omega)$
and the fact that the $\P$-tableau and $\Q$-tableau of a permutation have
 the same shape.
Since $\sh \SD(w)=\sh \P(w)$, Theorem~\ref{thm:tfae} tells us that
$\SD(w)$ is standard and thus
$w$ is good.
\end{proof}

Conjectures about good permutations and good tableaux are given in 
Section~\ref{sec:further directions}.

%%%%%%%
%%%%%%%
%%%%%%%

% \section{First soliton and carrier algorithm}
\section{First soliton is created after one BBS move}
\label{sec:first soliton}

In this section, we prove 
that applying one BBS move to a permutation $w$ is enough to obtain the first (rightmost) soliton of $\SD(w)$, and this first soliton is equal to the first row of $\P(w)$.
Our proof uses the carrier algorithm (explained below) and the localized Schensted's theorem (see Section~\ref{sec:local Schensted's theorem}).

\subsection{Carrier algorithm}

The carrier algorithm is a way to transform a BBS configuration at time $t$ into the configuration at time $t + 1$. 
In the algorithm, we move numbers in and out of a carrier in a way that is similar to the insertion rule of the RS algorithm.
A version of the carrier algorithm was first introduced in \cite{TM97carrier}, and the following definition comes from~\cite[Section 3.3]{Fuk04}.

\begin{definition}[Carrier algorithm]
Let $X$ be a BBS configuration with $n$ balls, that is, $X$ is a sequence indexed by $\mathbb{Z}$ where each of $1,\dots,n$ (denoting balls) appears exactly once, and $e \coloneqq n+1$ (denoting empty boxes) appears infinitely many times. 
Let the length-$n$ sequence $C = \underbracket{\, e \, e \, \dots \,\,}$ be the initial state of the carrier, and let $j$ be the smallest number such that $X(j)\neq e$.
Let $X'$ be a BBS configuration defined as follows. 
\begin{enumerate}[Step 1:]
\item
\label{itm:carrier algorithm:step1} 
\begin{enumerate}
\item 
\label{itm:X(j) is a ball}
If $X(j) < \max  C$, let $y$ be the smallest number in the carrier $C$ greater than $X(j)$.   
Set $X'(j) =y$.  
Remove $y$ out of $C$ and insert $X(j)$ into $C$. 
    
\item \label{itm:X(j) is an empty box}
If $X(j) \geq \max C$, let $m=\min C$.
Set $X'(j) = m$. 
Remove $m$ out of $C$ and insert $X(j)$ into $C$.  
\end{enumerate}

\item
Set $j \coloneqq j+1$.
If $X(k)\neq e$ for some $k \geq j$ or if $C$ still contains balls, repeat Step~\ref{itm:carrier algorithm:step1};  
Otherwise, we are done.

\end{enumerate}
\noindent Let $X'(i)=e$ for the rest of the boxes $i$ which have not been assigned a value.
\end{definition}

Note that, since we have exactly $n$ balls and the carrier can carry $n$ elements, we have
\begin{itemize}
\item 
 $X(j)\neq e$
 iff $C$ has a number greater than $X(j)$, which is case~\eqref{itm:X(j) is a ball} in Step~\ref{itm:carrier algorithm:step1},
 and 
 \item
$X(j) = e$ 
iff $C$ has no number greater than $X(j)$, which is case \eqref{itm:X(j) is an empty box} in Step~\ref{itm:carrier algorithm:step1}.
\end{itemize}

\begin{theorem}
[{\cite[Proposition 3.2]{Fuk04}}]
Running the carrier algorithm once is equivalent to performing a BBS move once.
That is, given a BBS state $X$ at time $t$, the BBS configuration $X'$ we get by performing the carrier algorithm is the state at time $t+1$.
\end{theorem}

\begin{remark}
\label{rem:carrier algorithm is the same as RS insertion algorithm for row 1}
When we insert a consecutive sequence of balls (for example, when $X$ comes from a permutation), the rule for bumping and inserting numbers into and out of the carrier is the same as the rule for bumping and inserting numbers into and out of the first row of the insertion tableau  during the RS algorithm. 
\end{remark}

\begin{example}
\label{ex:carrier:452361:t=2 to t=3}
We apply the carrier algorithm to the BBS configuration of the box-ball system from Figure~\ref{fig:intro:452361:t=0 to t=4} at time $t=2$
\[ 452ee136 \]
to obtain the configuration at time $t=3$. 
For the purpose of proving the main result of this section, 
it is helpful to break the carrier algorithm into two processes: the first process is to insert into the carrier $C$ all balls and the $e$'s between them. The second process is to ``flush" out all balls from $C$ by 
inserting enough $e$'s into it. 
\begin{multicols}{2}
\begin{minipage}
{0.42\textwidth}
\begin{align*}
\text{\textbf{begin}} &\text{ Process 1:} \text{ insert all balls}\\
&\underbracket{eeeeee}452ee136\\
e&\underbracket{4eeeee}52ee136\\
ee&\underbracket{45eeee}2ee136\\
ee4&\underbracket{25eeee}ee136\\
ee42&\underbracket{5eeeee}e136\\
ee425&\underbracket{eeeeee}136\\
ee425e&\underbracket{1eeeee}36\\
ee425ee&\underbracket{13eeee}6\\
ee425eee&\underbracket{136eee}\\
\text{\textbf{end}} & \text{ Process 1}%\text{ insertion of all balls}
\end{align*}
\end{minipage} 

\columnbreak

\begin{minipage} 
{0.42\textwidth}
\begin{align*} 
 \text{\textbf{begin}} \text{ Process 2:} &\text{ flushing process}\\
ee425eee&\underbracket{136eee}\leftarrow e\\
ee425eee1&\underbracket{36eeee}\leftarrow e\\
ee425eee13&\underbracket{6eeeee}\leftarrow e\\
ee425eee136&\underbracket{eeeeee}\\
\text{\textbf{end}} \text{ Process 2} &
\end{align*}
\end{minipage}
\end{multicols}
The sequence 
 $ee425eee136$ to the left of the carrier 
 corresponds to the configuration at time $t=3$ given in Figure~\ref{fig:intro:452361:t=0 to t=4}.
\end{example}

%%%
%%%
%%%
%%%

\subsection{First soliton}

We refer to the rightmost soliton in a steady-state configuration as the \emph{first soliton} of the box-ball system. 
The \emph{first soliton} of a permutation $w$ is the first soliton of the box-ball system containing $w$, that is, the first row of $\SD(w)$. 
Let $\Rowone(T)$ denote the  first row of a tableau $T$. 
The following result shows that the first soliton is created after applying at most one BBS move to a permutation $\omega$. 
Furthermore, the first row of the soliton decomposition of $w$ 
is equal to the first row of $\P(w)$.

\begin{theorem}
\label{thm:first soliton created after 1 BBM}
If $\omega$ is a permutation, then we have the following. 
\begin{enumerate}
\item \label{thm:itm:first soliton contains 1}
The first soliton $\Rowone(\SD(w))$ contains the ball $1$.
\item 
\label{thm:itm:first soliton created after 1 BBM}
The first soliton 
$\Rowone(\SD(\omega))$
of $w$ is created after at most one BBS move.
That is, the rightmost increasing run of $\BB^t(w)$ is equal to the first soliton of $w$ for all $t \geq 1$. 
\item 
\label{thm:itm:Rowone(SD(w))=Rowone(P(w))}
$\Rowone(\SD(\omega))=\Rowone(\P(\omega))$.
\end{enumerate}
\end{theorem}
\begin{proof}
Let $\omega=w_1 w_2 \dots w_n \in S_n$. 
If $\omega$ has steady-state time $0$ then the first soliton of $\omega$ is already created, so suppose that $\omega$ has steady state time $m \geq 1$.

For each time $t$, let $a_t$ denote the increasing run containing the ball 1 in the BBS configuration $\BB^t(\omega)$. 
We will prove that $\Rowone(\SD(\omega))$ is equal to $a_m = a_{m-1} = \dots = a_2 = a_1$.

We apply the carrier algorithm to $\omega$. 
Insert all the balls $w_1, w_2, \dots, w_n$ of $\omega$ into the carrier, and pause immediately after the last ball $\omega_n$ of $\omega$ is inserted into the carrier. 
Let $c$ denote the sequence of balls which is currently in the carrier.
Since $\omega$ is a permutation,  the rule for bumping and inserting numbers into and out of the carrier is the same as the rule for bumping and inserting numbers into and out of the first row of the insertion tableau during the RS algorithm (see Remark~\ref{rem:carrier algorithm is the same as RS insertion algorithm for row 1}). Therefore, %ignoring the copies of $e$'s in the carrier, 
we have 
\[
c=\Rowone(\P(\omega)). 
\]
In particular, $c$ is an increasing sequence starting with the value $1$.
When we flush $\text{len}(c)$ copies of $e$ into the carrier to finish the carrier algorithm, the sequence $c$ is the rightmost $\text{len}(c)$ letters of $\BB(\omega)$.
Thus $c$ contains the value $1$ and is the rightmost increasing run of $\BB(\omega)$.
Therefore, $c=a_1$, and so 
\begin{align}
 \nonumber &\text{$a_1$ is the rightmost increasing run of $\BB(\omega)$} 
\end{align}
and
\begin{align}
a_1 = \Rowone(\P(\omega)).\label{eq:a1=rowone(P(w))}
\end{align}
Schensted's theorem (Theorem~\ref{thm:Schensted's theorem}) tells us that the size of 
$\Rowone(\P(\omega))$
is equal to $\incr(\omega)$, the length of a longest increasing subsequence of $\omega$, so 
\begin{equation}
\label{eq:len(a1)=i(w)}
\text{len}(a_1) = \incr(\omega).
\end{equation}

Since $a_1$ is the rightmost increasing run of $\BB(\omega)$ and $a_1$ starts with the value $1$, applying a BBS move to $\BB(\omega)$ will produce a rightmost increasing run containing $a_1$. 
So the increasing run $a_2$ of $\BB^2(\omega)$ containing $1$ is the rightmost increasing run of $\BB^2(\omega)$ and $a_2 \supseteq a_1$. 
By the same reasoning, applying a BBS move to $\BB^2(\omega)$ will produce a rightmost increasing run  containing $a_2$, so the increasing run $a_3$ of $\BB^3(\omega)$ containing $1$ is the rightmost increasing run of $\BB^3(\omega)$ and $a_3 \supseteq a_2$, and so on.
Therefore, the increasing run $a_m$ is the rightmost increasing run of $\BB^m(w)$.
Since $\BB^m(w)$ is in steady state, 
\[\text{
$a_m$ is the first soliton of $w$,} \]
proving part~\eqref{thm:itm:first soliton contains 1}. 
In addition, 
\begin{equation*}
a_m \supseteq \dots \supseteq a_2 \supseteq a_1.
\end{equation*} 
To prove $a_m=a_1$, we proceed to show that $\len(a_t) \leq 
\len(a_1)$ for $t \geq 1$.
First note that the penalized length of $a_t$ is $\len(a_t)$ because $a_t$ is an increasing run. 
Since $\localincr(\BB^t(\omega))$ is defined to be the maximum penalized length over all increasing sequences of balls in $\BB^t(\omega)$, we have 
\begin{align*}
\len(a_t) &\leq \localincr(\BB^t(\omega)) \\
&= \incr(\omega)
\text{ by Theorem~\ref{thm:local Schensted's theorem for BBS}}
\\
& = \len(a_1) \text{ by \eqref{eq:len(a1)=i(w)}}
\end{align*}
Thus, $a_m = a_1$.
Since $a_1$ is the rightmost increasing run of $\BB(\omega)$ and $a_m$ is the first soliton of $w$, this equality concludes the proof of  part~\eqref{thm:itm:first soliton created after 1 BBM} of the theorem.

Finally, we have $a_m=a_1=\Rowone(\P(\omega))$ by~\eqref{eq:a1=rowone(P(w))}, proving part~\eqref{thm:itm:Rowone(SD(w))=Rowone(P(w))}.
\end{proof}
%%%
%%%
%%%

\begin{corollary}
\label{cor:k+1 st soliton created after 1 BBM in special case}
Let $\omega$ be a permutation and suppose that the $k$ rightmost solitons of $\omega$ are already formed.  
Then it takes at most one BBS move to create the ${(k+1)}^{\text{th}}$ 
rightmost soliton of $\omega$. 
\end{corollary}

\begin{remark}
Corollary~\ref{cor:k+1 st soliton created after 1 BBM in special case} does not hold if we replace $\omega$ with a BBS configuration that has empty boxes between balls. For example, consider the configuration 
\begin{equation*}
    ee45e2136
\end{equation*}
at time $t=1$ in Figure~\ref{fig:intro:452361:t=0 to t=4}, which is not a permutation. 
In this configuration the first soliton $136$ has already been created.
However, the second soliton $25$ is not created until after two BBS moves later.
\end{remark}

%%%%
%%%%
%%%%

\section{L-shaped soliton decompositions}
\label{sec:L-shaped}

In this section, we prove that permutations with L-shaped soliton decompositions have steady-state time at most 1. 
We also study noncrossing involutions, nested involutions, and column reading words. We prove that these involutions all have L-shaped soliton decompositions and therefore have steady-state time at most $1$.

\subsection{L-shaped soliton decompositions} 
Let $\SST(w)$ denote the steady-state time of a permutation $w$.

\begin{theorem} \label{thm: L-shaped SD SST}
If a permutation $w$ has an L-shaped soliton decomposition, that is, the partition $\sh\SD(w)$ is of the form $(\incr,1,1,\dots)$, then $\SST(w) \leq 1$.
\end{theorem}

\begin{example}
\label{ex:thm: L-shaped SD SST:5274163}
% Example of an involution which is not noncrossing 
Let $\pi=5274163$. Applying the first BBS move, we get \[e5e72e4136\] which is in steady state with soliton decomposition 
\[
SD(\pi)=\young(136,4,2,7,5)\]
\end{example}

\begin{proof}[Proof of Theorem \ref{thm: L-shaped SD SST}] 
Let $h \coloneqq n-\incr$, so the number of rows of $\SD(\omega)$ is $1+h$. 
Theorem~\ref{thm:first soliton created after 1 BBM} tells us that the rightmost increasing run of $\BB^{1}(\omega)$
is equal to the first soliton.  
Since the number of increasing runs of a BBS configuration is preserved by a BBS move (Remark~\ref{rem:height of ID is an invariant}), the number of rows of $\ID(\BB^1(w))$ is $1+h$. 
So the shape of $\ID(\BB^1(w))$ is equal to $(\incr, 1,1,\dots)$ and $\BB^1(\omega)$ is of the form 
\[X =
b_h \dots b_{h-1} \dotsc b_2 \dots b_1 \dots ~
\underbrace{1 s_2 s_3 \dotsc}_{\text{first soliton}},
\] 
such that, for each $b_j \in \{b_1, \dots, b_h\}$,  either
\begin{enumerate}
\item 
there is an empty box immediately to the left of $b_j$, or
\item $b_{j+1} b_j$ is a  consecutive, decreasing subsequence of $X$.
\end{enumerate}
Therefore, the configuration array of $\BB^1(w)$ is a standard skew tableau whose row sizes are weakly increasing,   
so $\BB^1(w)$ is in steady state by Proposition~\ref{prop:t=0 generalization}. 
\end{proof}

%%%
%%%
%%%

Next, we point out a characterization of permutations with L-shaped soliton decompositions.

\begin{lemma}
\label{lemma for prop:noncrossing have L-shaped SD}
Suppose $w$ is a  permutation in $S_n$. 
Let $\incr$ denote the length of a longest increasing subsequence of $w$  and $\des$ the number of descents of $w$. 
Then $\SD(w)$ is L-shaped iff $\incr + \des \geq  n$.
In this case, $\SD(w)$ has shape $(\incr,\underbrace{1,1,\dots,1}_{\des~\text{copies}})$.
\end{lemma}
\begin{proof}
The localized Schensted's theorem (Theorem~ \ref{thm:local Schensted's theorem}) tells us that the length of the first soliton is $\incr$.
It also tells us that the length of the first column of the soliton decomposition is $\des+1$.
Since $\SD(w)$ has size $n$, it must be that $\SD(w)$ is L-shaped iff $\incr + \des = n$. 
Furthermore, since $\incr + \des \leq n$ holds for all permutations in $S_n$, writing $\incr + \des = n$ is equivalent to writing $\incr + \des \geq n$.
\end{proof}

\subsection{Noncrossing involutions have L-shaped soliton decompositions}

\begin{definition}[Noncrossing involution]
A pair of distinct 2-cycles is called a \emph{crossing} if they can be written as $(ac)$ and $(bd)$ where 
$a<b<c<d$. 
An involution is called \emph{noncrossing} if no pair of 2-cycles is a crossing.
\end{definition}

For example, the involution with cycle notation $(26)(34)(78)$ is noncrossing, but the involution with cycle notation $(24)(36)(78)$ is not noncrossing, since $(24)$ and $(36)$ is a pair of crossing 2-cycles. 
Any 2-cycle is a noncrossing involution, as is the identity permutation.

\begin{proposition}
\label{prop:noncrossing have L-shaped SD}
If $\omega$ is a noncrossing involution, then $\omega$ has an L-shaped soliton decomposition. 
More precisely, let $w$ be a noncrossing involution in $S_n$, let $c$ denote the number of adjacent $2$-cycles of $w$, and let $k$ denote the number of all $2$-cycles of $w$ (including the adjacent $2$-cycles). 
Then the shape of $\SD(w)$ is $(n-2k+c,\underbrace{1,1,\dots,1}_{2k - c~\text{copies}})$.
\end{proposition}

The following example illustrates the idea of our proof of Proposition~\ref{prop:noncrossing have L-shaped SD}.
\begin{example}
Let $w=164352879=(26)(34)(78)\in S_9$.
First, we construct an increasing subsequence of the one-line notation of $w$. 
Since $w$ has three $2$-cycles, we know that $w$ has exactly $9-2(3)=3$ fixed points: $1$, $5$, and $9$.
These three fixed points form an increasing subsequence of $w$. 
We have two adjacent $2$-cycles $(34)$ and $(78)$, and we can add $3$ and $7$ to $1\,5\,9$ to form an increasing subsequence of $w$ of size five: $1\,3\,5\,7\,9$. So the size of the first soliton of $w$ is at least $5$ by the localized Schensted's theorem (Theorem~ \ref{thm:local Schensted's theorem}).

Next, we look for the descents of $w$.
The one non-adjacent $2$-cycle $(26)$ contributes two descents $2$ and $6-1=5$ to $w$, since $w(2)=6>4=w(3)$ and 
$w(5)=5>2=w(6)$.
The two adjacent $2$-cycles $(34)$ and $(78)$ contribute one descent each to $w$ %, $3$ and $7$, 
because $w(3)=4>3=w(4)$ and $w(7)=8>7=w(8)$.
We have found at least $4$ descents of $w$.
So the size of the first column is at least $4+1=5$ by the localized Schensted's theorem (Theorem~ \ref{thm:local Schensted's theorem}).

The size of $\SD(w)$ is $9$, so its shape must be $(5,1,1,1,1)$. Indeed, 
\[\SD(w)=\young(12579,8,3,4,6)\]
\end{example}

\begin{proof}[Proof of Proposition~\ref{prop:noncrossing have L-shaped SD}]
Let $w$ be a noncrossing involution in $S_n$ which is not the identity permutation, and let $k\geq 1$ denote the number of all $2$-cycles of $w$.
First, we construct an increasing subsequence of the one-line notation of $\omega$. 

Since the only values changed by $\omega$ are the ones in the $2$-cycles, $\omega$ has $n-2k$ fixed points.
First, consider the case where $n>2k$, so that $w$ indeed has fixed points. 
The $n-2k$ fixed points of $\omega$  form an increasing subsequence $a_1 a_2\dots a_{n-2k}$ of $w$.

Let $c \geq 0$ be the number of adjacent $2$-cycles of $w$, and consider the adjacent $2$-cycles of $w$ 
\[
(i_1, i_1+1),
(i_2, i_2+1),
\dots,
(i_c, i_c+1),
\]
listed from smallest to largest, that is, $i_1 < i_2 < \dots < i_c$.
Note that each of the adjacent $2$-cycles simply swaps $i_j$ and $i_j+1$,  
so $i_1 i_2 \dots i_c$ is an increasing  subsequence of $w$.
Furthermore, if $w$ has fixed points we can insert $i_1, i_2, \dots, i_c$ into the increasing subsequence $a_1 a_2 \dots a_{n-2k}$ of $w$ to form a longer subsequence of $w$ of size $n-2k+c$. 
Let $\incr$ denote the size of a longest increasing subsequence of $w$;  we have shown that $\incr \geq n-2k+c$.

Next, let's compute the number of descents of $\omega$. 
First, consider a non-adjacent 2-cycle $(xz)$ where $x+1<z$. 
We claim that 
\begin{center} $x$ is a descent of $w$.
\end{center}
Either $x+1$ is a fixed point or $x+1$ is part of a $2$-cycle.
If $x+1$ is a fixed point, then $w(x+1)=x+1$ and we have $w(x)=z>x+1=w(x+1)$, so $x$ is a descent of $w$. 
If $x+1$ is part of a $2$-cycle 
$(x+1,y)$, 
then $y$ must be smaller than $z$ because $w$ is a noncrossing involution. Therefore, $w(x)=z>y=w(x+1)$, so again $x$ is a descent of $w$.
Using a similar argument, we can show that 
\begin{center}
$z-1$ is a descent of $w$.    
\end{center}
% [Proof: Either $z-1$ is a fixed point or $z-1$ is part of a $2$-cycle. If $z-1$ is a fixed point, then $w(z-1)=z-1>x=w(z)$, so $z-1$ is a descent of $w$. If $z-1$ is part of a $2$-cycle  $(y',z-1)$, then $y'$ must be bigger than $x$ because $w$ is a noncrossing involution. Therefore, $w(z-1)=y'>x=w(z)$, so $z-1$ is a descent of $w$.] 

For each adjacent $2$-cycle $(x,x+1)$,
\begin{center}
the number $x$ is a descent of $w$
\end{center}
because $w(x)=x+1>x=w(x+1)$.
In total, we have shown that $w$ has at least $2k-c$ descents. If we let $\des$ denote the number of descents of $w$, we have $\des \geq 2k-c$.

We have shown that $\incr \geq n-2k+c$ and that $\des \geq 2k - c$. 
Since $(n-2k+c) + (2k - c)=n$, we have $\incr + \des \geq n$, so $\SD(w)$ is L-shaped 
with  
 shape $(\incr,\underbrace{1,1,\dots,1}_{\des~\text{copies}})$
by 
Lemma~\ref{lemma for prop:noncrossing have L-shaped SD}.
\end{proof}

\begin{remark}
\label{rem:crossing involution with shape L SD:5274163}
% Example of an involution which is not noncrossing and has SD which is L-shaped
Not all involutions with 
L-shaped soliton decompositions 
are noncrossing involutions. 
For instance, the involution $\pi=5274163=(15)(37)$ from Example~\ref{ex:thm: L-shaped SD SST:5274163}
 has a crossing and has an L-shaped soliton decomposition.
\end{remark}

The following result is a consequence of Theorem \ref{thm: L-shaped SD SST} and Proposition~\ref{prop:noncrossing have L-shaped SD}.
\begin{corollary} 
\label{cor:noncrossing steady state}
All noncrossing involutions have steady-state time at most $1$. 
\end{corollary}

Two families of tableaux that correspond to noncrossing involutions are discussed next.

%%%
%%%
%%%

\subsection{Nested involutions have L-shaped soliton decompositions}

\begin{definition}
\label{def:nested} 
A pair of distinct 2-cycles is called a \emph{nesting} if they can be written as 
$(ad)$ and $(bc)$
 where 
$a<b<c<d$.  
An involution is called \emph{nested}
if every pair of 2-cycles is a nesting.
\end{definition}

\begin{example}
Any 2-cycle is a nested involution, as is the identity permutation. 
The involutions $(15)(24)$ and $(17)(25)(34)$ are nested involutions, but $(23)(45)(17)$ is not because the pair $(23)$ and $(45)$ is not a nesting.
\end{example}

\begin{corollary}
If $w$ is a nested involution then $\SST(w) \leq 1$. 
\end{corollary}
\begin{proof}
Since a nested involution is a noncrossing involution, by Corollary~\ref{cor:noncrossing steady state} its steady-state time is at most $1$.
\end{proof}

The following lemma is a special case of \cite[Theorem 5.2]{Pos09}. The forward direction of the lemma can be proven by applying the inverse RS algorithm, and the reverse direction by Schensted's theorem (Theorem~\ref{thm:Schensted's theorem}).

\begin{lemma} 
\label{lem:L shape iff w has nested cycles}
Suppose $w$ is an involution. 
Then 
the RS partition 
of $w$ is
L-shaped iff $w$ is a nested involution. 
\end{lemma}

The following tells us that nested involutions are (BBS) good, but all other noncrossing involutions are not good.
\begin{proposition}
\label{prop:noncrossing w is good iff w is nested}
Suppose $w$ is noncrossing. Then $w$ is good iff $w$ is a nested involution.
\end{proposition}
\begin{proof}
Suppose $\omega$ is noncrossing. 
By Proposition~\ref{prop:noncrossing have L-shaped SD}, the BBS soliton partition of $\omega$ is L-shaped.
The permutation $\omega$ is good iff the RS partition of $\omega$ is equal to the BBS soliton partition of $\omega$ (Theorem~\ref{thm:tfae}). 
This equality holds iff the RS partition of $w$ is L-shaped, which is true iff $w$ is a nested involution (Lemma~\ref{lem:L shape iff w has nested cycles}).
\end{proof}

\begin{remark}
The previous proposition implies that if an involution is good, then it either has a crossing or it is a nested involution. (The converse is false: we can find an involution which has a crossing but is not good. 
For instance, the involution $\pi=5274163=(15)(37)$ from Example~\ref{ex:thm: L-shaped SD SST:5274163} has a crossing and is not good.)
\end{remark}

% The set of bad involutions in $S_6$ which have a crossing. There are six of them.
% [(1, 3), (2, 4), (5, 6)]
% [(1, 2), (3, 5), (4, 6)]
% [(1, 3), (2, 5), (4, 6)]
% [(1, 4), (3, 6)]
% [(1, 4), (2, 6)]
% [(1, 5), (3, 6)]

% \begin{remark} Suppose $\omega$ is a good involution. If $\omega$ is noncrossing, then $\omega$ is a nested involution. \end{remark}

\subsection{Column reading words have L-shaped soliton decompositions}

Remark~\ref{rem:t=0} tells us that a permutation has steady-state time $0$ iff it is the row reading word of a standard tableau.  
In this section, we prove that the column reading word of a standard tableau has steady-state time at most $1$. 

\begin{definition}
The \emph{column reading word} or \emph{column word} of a tableau $T$ is the word obtained by reading the columns of $T$ bottom to top,
from left to right.
\end{definition}

If $w$ is the column word of a standard tableau $T$, then $\P(w)=T$ (see~\cite[Section 2.3]{fulton_1996}). 
For instance, $63174285$ is the column word of the standard tableau
\[T=\young(125,348,67)=\P(63174285).\]

\begin{definition}
The \emph{column superstandard} tableau of shape $\lambda$ is the tableau of shape $\lambda$ which is obtained by filling the columns top to bottom, from left to right, with the integers $1,2,3,\dots, n$, in this order.
\end{definition}

The following lemma can be deduced from applying the inverse RS algorithm and from the fact that all column reading words of standard tableaux of the same shape are dual Knuth equivalent. 
The forward direction of the lemma is stated in~\cite[Lemma 4.2]{Gill21}.
\begin{lemma}
\label{lem:column word superstandard}
A permutation $w$ is the column reading word of a standard tableau iff $\Q(w)$ is column superstandard.
\end{lemma}

For example, 
let $w=63174285$ be the column word 
from the previous example. We have
\[ \P(w)=\young(125,348,67) \hspace{5mm} \Q(w)=\young(147,258,36)\]
where $\Q(w)$ is the column superstandard tableau of shape $(3,3,2)$.

\begin{remark} \label{rmk: superstandard to cycle}
If $T$ is a column superstandard tableau, the one-line notation of the 
involution $\pi$ where $\P(\pi)=\Q(\pi)=T$ 
is the column word of $T$. 
Equivalently, the cycle notation for $\pi$
can be described as follows.
Take each column of $T$ and ``fold" it in the middle. Each pair of entries that touch gives us a 2-cycle of $\pi$, and the entry in the center of the column (if the column has odd length) gives us a fixed point of $\pi$.
Therefore, $\pi$ is a noncrossing involution. If the second column has length at least $2$, then $\pi$ is not a nested involution (see Definition~\ref{def:nested}). 
For example, consider $\pi \in S_9$ where 
\[\P(\pi)=\Q(\pi)=\begin{ytableau} \textcolor{red}{\bf 1} & \textcolor{forGreen}{\bf 6} & \bf 9 \\ \textcolor{blue}{\bf 2} & \bf 7 \\ \bf 3 & \textcolor{forGreen}{\bf 8} \\ \textcolor{blue}{\bf 4} \\ \textcolor{red}{\bf 5}
\end{ytableau}.
\]
Then $\pi$ is the column word $54321 876 9$ in one-line notation and 
$\pi=\textcolor{red}{\bf (15)} \textcolor{blue}{\bf (24)} ~ {\bf (3)} ~\textcolor{forGreen}{\bf (68)} {\bf (7)} ~ {\bf (9)}$
in cycle notation, so $\pi$ is a noncrossing involution which is not nested.
\end{remark}

\begin{proposition} 
If $w$ is the column reading word of a standard tableau (equivalently, if $\Q(w)$ is a column superstandard tableau), then the steady-state time of $w$ is $0$ or $1$.
\end{proposition}
\begin{proof}
Let $w$ be the column reading word of a standard tableau (equivalently, $\Q(w)$ is a column superstandard tableau).  
Let $\pi$ be the involution
such that $\Q(\pi)=\Q(w)$. 
Since $\Q(\pi)$ is column superstandard,
Remark \ref{rmk: superstandard to cycle} tells us that $\pi$ is a noncrossing involution.  
Therefore, we have $\SST(\pi) \leq 1$ by 
Corollary \ref{cor:noncrossing steady state}. 
Since the recording tableau of a permutation determines steady-state time (Theorem~\ref{thm:Q determines steady state}), we have $\SST(w)=\SST(\pi) \leq 1$.
\end{proof}

\section{A maximum steady-state time}\label{sec:maximum sst}

The following theorem and conjecture are given in~\cite{DGGRS23}.
If $n\geq 5$, let
\[\widehat{\Q}_n= %\tiny
\begin{ytableau}
 1 & 2 & 5 & \none[\hdots]&\scalebox{.6}{$n-2$}&\scalebox{.6}{$n-1$}\\
 3 & 4\\
 n
 \end{ytableau}\]

\begin{theorem}[{\cite[Theorem~6.7]{DGGRS23}}] 
If $n \geq 5$ and $\Q(\omega)=\widehat{\Q}_n$, then the steady-state time of $\omega$ is $n-3$. 
\end{theorem}

\begin{conjecture}[{\cite[Conjecture~1.1]{DGGRS23}}]
\label{conjecture:n-3}
Let $n\geq 5$ and $\omega \in S_n$. If $\Q(\omega)$ is not equal to $\widehat{\Q}_n$,
then the steady-state time of $\omega$ is less than $n-3$.
\end{conjecture}

Since the recording tableau of a permutation determines its steady-state time (Theorem~\ref{thm:Q determines steady state}), if $T$ is a standard tableau, we can define the
\emph{steady-state time of $T$} 
to be the steady-state time for all permutations $w$ such that $\Q(w)=T$. 
Let $\SST(T)$ denote the steady-state time of a standard tableau $T$. 
In Section~\ref{sec:tableaux of shape (n-3,2,1)}, 
we prove a partial result: 
the maximum steady-state time 
for tableaux of shape $(n-3,2,1)$ is $n-3$. 
In Section~\ref{sec:chain of tableaux},
we present a chain of tableaux that have steady-state times from 0 to $n-3$.

\subsection{Maximum steady-state time for tableaux of shape (n-3, 2, 1)}
\label{sec:tableaux of shape (n-3,2,1)}
\begin{lemma}
\label{lem:sh(n-3,2,1)}
If $\omega \in S_n$ and
\begin{center}$\sh(\P(\omega))=(n-3,2,1)=\scalebox{.75}{\begin{ytableau} ~& & & & \none[\hdots] \\ & \\ \\ \end{ytableau}}$ 
\end{center}
then either 
$\sh(\SD(\omega))$ is $(n-3,2,1)$ or  $(n-3,1,1,1)$.
\end{lemma}
\begin{proof}
The fact that the size of the first row of $\SD(w)$ is $n-3$ follows from Remark~\ref{rem:sizes of RS and BBS}\eqref{itm:rem:lambdaRS1 equals lambdaBBS1}.
The size of the first column of $\SD(w)$ is at least $3$ by Remark~\ref{rem:sizes of RS and BBS}\eqref{itm:rem:perm stats inequalities}.
\end{proof}

%%%%
%%%%%

\begin{lemma}
\label{lemma:thm:max SST of T with sh(n-3,2,1)}
Let $n \geq 5$ and $\omega \in S_n$. Suppose that  
at time $t \geq 1$ we have the 
(non-steady-state)
 BBS configuration 
\begin{equation}
\label{eq:thm:max SST of T with sh(n-3,2,1):BB1:lemma}
X= \BB^t(\omega)=
\dots 
~ 
\underbrace{\mathbf{a} \, \mathbf{b}}_{\text{increasing run}}
~ \overbrace{e e \dots e}^{\substack{\text{$m$ copies}\\\text{of $e$'s}}}
~ \mathbf{x} ~
\overbrace{e \dots e}^{\substack{\text{$\ell$ copies}\\\text{of $e$'s}}} 
~ \underbrace{1 \, y_2 \, y_3 \, \dots \, y_{n-3}}_{\text{increasing run}} 
~
\dots 
\end{equation} 
where
\begin{itemize}
\item $\mathbf{a} < \mathbf{b}$
 is an increasing run
and 
$\mathbf{b} > \mathbf{x}$, 
\item $1 < y_2 < y_3 < \dots < y_{n-3}$ is the rightmost increasing run,  
\item 
 $m \geq 0$ is the number of empty boxes between $\mathbf{b}$ and $\mathbf{x}$.
\end{itemize}
Then we have the following.
\begin{enumerate} 
\item
$X$ first reaches steady state after we apply $m+1$ additional BBS moves; that is,  
$\BB^{m}(X)$ is not a steady-state configuration, but $\BB^{m+1}(X)$ is. In other words, $\SST(\omega)=t+m+1$.

\item 
If $\mathbf{a} < \mathbf{x}$, then 
$
\SD(\omega)=\begin{ytableau}
1 & y_2 & y_3 &\none[\dots]\\
\mathbf{a} & \mathbf{x} \\
\mathbf{b}
\end{ytableau}$;
otherwise, 
$
\SD(\omega)=\begin{ytableau}
1 & y_2 & y_3 & \none[\dots]\\
\mathbf{x} & \mathbf{b} \\
\mathbf{a}
\end{ytableau}.
$
In either case, $\SD(\omega)$ is standard, that is, $\omega$ is a (BBS) good permutation.
\end{enumerate}
\end{lemma}
\begin{proof}
By Theorem~\ref{thm:first soliton created after 1 BBM}, the rightmost increasing run $1 y_2 y_3 \dots y_{n-3}$ is the first soliton.

If $m>0$, we apply $m$ additional BBS moves to $X$.
At each BBS move,
the first soliton will move forward $n-3 \geq 2$ boxes and 
the increasing block $\mathbf{a} \mathbf{b}$ will move forward 2 boxes, and the singleton block $\mathbf{x}$ will move forward 1 box, 
so that the number of spaces between  $\mathbf{a} \mathbf{b}$ and $\mathbf{x}$ 
decreases by 1 after each BBS move.
The two blocks $\mathbf{a} \mathbf{b}$ and $\mathbf{x}$ touch in the configuration 
\[\BB^{m}(X)=\dots \mathbf{a} \, \mathbf{b} \,  \mathbf{x} \dots \] 
Since $\mathbf{x}< \mathbf{b}$, 
we have
\begin{align*}
\BB^{m+1}(X)=
\begin{cases}
\dots \mathbf{b} \overbrace{ \mathbf{a} \mathbf{x} }^{\text{soliton}} \dots & \text{ if $\mathbf{a} < \mathbf{x}$}\\
\dots \mathbf{a} \underbrace{\mathbf{x} \mathbf{b} }_{\text{soliton}} \dots & \text{ if $\mathbf{x} < \mathbf{a}$}
\end{cases}
\end{align*}

Let $T$ be the configuration array of $\BB^{m+1}(X)$ (see Section~\ref{sec:configuration array}).
If there is at least one empty box between these three balls and the first soliton, 
the inequalities involving the numbers $a, b, x, 1,$ and $y_2$ 
guarantee that
$T$ is a standard skew tableau whose rows are weakly decreasing in length. 
If there is no gap between these three balls and the first soliton, we must have $\omega \in S_5$ where 
\begin{align*}
\BB^{m+1}(\omega)=
\begin{cases}
\dots \mathbf{b} \overbrace{ \mathbf{a} \mathbf{x} }^{\text{soliton}} 1 y_2
& \text{ if $\mathbf{a} < \mathbf{x}$}
\\
\dots \mathbf{a} \underbrace{\mathbf{x} \mathbf{b} }_{\text{soliton}} 1 y_2
& \text{ if $\mathbf{x} < \mathbf{a}$}
\end{cases}
\end{align*}
If $\mathbf{a} < \mathbf{x}$, we claim that $y_2<x$. Otherwise, we would have 
$\mathbf{a} < \mathbf{x} < y_2$, making $\localincr(\BB^{m+1}(\omega)) \geq 3$, contradicting the fact that $1 y_2$ is the first soliton. 
By similar argument, if $\mathbf{x} < \mathbf{a}$, we must have $y_2<b$.

Therefore, $T$
is a standard skew tableau whose rows are weakly decreasing in length. Thus $\BB^{m+1}(X)$
is in steady state by Proposition~\ref{prop:t=0 generalization}.
Since the order that the balls appear in $\BB^{m}(X)$ is different than in $\BB^{m+1}(X)$, we know that $\BB^{m}(X)$ is not yet in steady state. 
\end{proof}
%%%%
%%%%

\begin{theorem}
\label{thm:max SST of T with sh(n-3,2,1)}
If the RS partition of $w$ is 
 $(n-3,2,1)$, then $\SST(w) \leq n-3$. 
\end{theorem}
\begin{proof}
Suppose $\omega \in S_n$ and with RS partition $(n-3,2,1)$.  
Lemma~\ref{lem:sh(n-3,2,1)} tells us that $\sh(\SD(\omega))$ is either $(n-3,1,1,1)$ or $(n-3,2,1)$. 
If $\sh(\SD(\omega))=(n-3,1,1,1)$, then by Theorem \ref{thm: L-shaped SD SST} we have $\SST(\omega) \leq 1$. 
So suppose we have 
\begin{equation}
\label{eq:thm:max SST of T with sh(n-3,2,1):shape of SD}
\sh(\SD(\omega)) %=\sh(\P(\omega))
=(n-3,2,1).
\end{equation} 

At time $t=0$, let the $n$ balls $w_1 w_2 \dots w_n$ of $w$ be in boxes $1$ through $n$. 
We apply one BBS move to $\omega$ and consider all possibilities for the configuration $\BB^1(\omega)$ at time $t=1$. 
By Theorem~\ref{thm:first soliton created after 1 BBM}, we know the first soliton has been formed by $t=1$, 
so we only need to consider the possibilities for the remaining three balls.
By Remark~\ref{rem:height of ID is an invariant}, the number of rows in $\ID(\BB^1(\omega))$ is equal to that of $\SD(\omega)$, so $\ID(\BB^1(\omega))$ has three rows. 
Thus, the remaining three balls 
form a length-$2$ increasing run $\textbf{ab}$ and a length-$1$ (singleton) increasing run $\textbf{x}$.

If the length-$1$ block is to the left of the length-$2$ block at $t=1$, 
then $\BB^1(\omega)$ is already in steady state 
because Theorem~\ref{thm:first soliton created after 1 BBM} tells us that the rightmost soliton won't interact with the three balls after $t=1$. Therefore, $\SST(\omega) \leq 1$.

So suppose 
the length-$2$ block is to the left of 
the length-$1$ block at $t=1$, that is, 
\begin{equation}
\label{eq:thm:max SST of T with sh(n-3,2,1):BB1}
\BB^1(\omega)=\underbrace{
\overbrace{e \dots e}^{\substack{\text{$k$ copies}\\\text{of $e$'s}}}
~ 
\mathbf{a} \, \mathbf{b}  ~ \overbrace{e e \dots e}^{\substack{\text{$m$ copies}\\\text{of $e$'s}}}
~ \mathbf{x} ~
\overbrace{e \dots e}^{\substack{\text{$\ell$ copies}\\\text{of $e$'s}}}}_{\text{$n$ boxes}} ~ 
\underbrace{1 \, y_2 \, y_3 \, \dots y_{n-3}}_{\text{first soliton}} 
\end{equation}
where
\begin{itemize}
\item $\mathbf{a} < \mathbf{b}$
\item $\mathbf{a}$ is in box $k+1$ 
\item 
 $m \geq 0$ is the number of empty boxes between $\mathbf{b}$ and $\mathbf{x}$,
\item $\ell \geq 0$ is the number of empty boxes between $\mathbf{x}$ and the ball $1$.
\end{itemize}

First, observe that $\mathbf{x}< \mathbf{b}$. Otherwise, we would have $\mathbf{a} < \mathbf{b} < \mathbf{x}$, and eventually 
the increasing run $\mathbf{ab}$ would catch up to $\mathbf{x}$, forming a length-3 soliton $\mathbf{abx}$; this would mean that
$\sh(\SD(\omega))=(n-3,3)$, contradicting~\eqref{eq:thm:max SST of T with sh(n-3,2,1):shape of SD}.
Thus, 
$\BB^1(\omega)$
is of the form~\eqref{eq:thm:max SST of T with sh(n-3,2,1):BB1:lemma} in Lemma~\ref{lemma:thm:max SST of T with sh(n-3,2,1)}, so 
\[\SST(\omega)=1+m+1=m+2.\]

Finally, observe that $k \geq 2$ because $\mathbf{a} < \mathbf{b}$. 
We also know that  
 $k+m+3+\ell = n$ because ball $1$ is in box $n+1$ at time $1$.
Putting these together, 
we have \begin{align*}
m&=n-k-3-\ell \leq n-2-3,\\
m+2 &\leq n-3,
\end{align*}
proving $\SST(\omega)\leq n-3$.
\end{proof}

\subsection{Tableaux with increasing steady-state times via Bender--Knuth involution}
\label{sec:chain of tableaux}

In this section, we 
create a sequence of $n-2$ 
good tableaux whose steady-state times are from $0$ to $n-3$.

\begin{definition}[Bender--Knuth involution]
Let $T$ be a standard tableau with shape $\lambda$ and size $n$. Then for each $i \in \{1, \dots, n-1\}$, $\sigma_i$ is a map from the set of all standard tableaux of shape $\lambda$ to itself. 
If the $i$, $i+1$ are in the same row or column, then $\sigma_i(T) = T$. Otherwise, the map $\sigma_i$ swaps the numbers $i$ and $i+1$ in $T$. 
\end{definition}
\begin{example}
For instance,
\[T=\young(136,25,4) ~ \neq  ~
\sigma_2(T)=\young(126,35,4)
\]
but 
$\sigma_3(\sigma_2(T)) = \sigma_2(T)$.
\end{example}

Corollary~\ref{cor:noncrossing steady state}
and 
Proposition~\ref{prop:noncrossing w is good iff w is nested}
tell us that all noncrossing involutions have steady-state time $0$ or $1$ and
that most noncrossing involutions are bad. 
In combinatorics, the ``noncrossing" objects and the ``nonnesting" objects are often equinumerous, so it is natural to ask for a nonnesting analog of 
these results. 
The next proposition 
tells us that, for $t \in \{0, \dots,  n-3 \}$, there is a ``nonnesting" involution $W_t$ with steady-state time $t$.

\begin{proposition}
\label{prop:chain of Q with SST 0 to n-3}
Let
\begin{equation*}
Q_0 \coloneqq Q_0(n)=
\begin{ytableau} 
1 & 3 & 6 & 7 & 8 & \hdots & n \\
2 & 5 \\
4
\end{ytableau}
\end{equation*}
Then we have the following.
\begin{itemize}
\item $\SST(Q_0)=0$
\item $\SST(\sigma_2(Q_0))=1$
\item $\SST(\sigma_k \dots \sigma_5 \sigma_4 \sigma_2(Q_0))=k-2$, for each $k=4,5,\dots,n-1$.
\end{itemize}
Furthermore, each tableau in the sequence of tableaux 
\[
\text{$Q_0$,  $\sigma_2(Q_0)$, and
$\sigma_k \dots \sigma_5 \sigma_4 \sigma_2(Q_0)$}
\]
is (BBS) good.

\end{proposition}
\begin{proof}
The involution $W_0=\RSK^{-1}(Q_0, Q_0)$ is $(14)(35)$ in cycle notation and $4251367\dots n$ in one-line notation.
Since the latter %$4251367\dots n$ 
is the row reading word of a standard tableau (namely, $Q_0$), 
Remark~\ref{rem:t=0} tells us that 
$W_0$ 
has steady-state time 0.

Next, consider 
% $\sigma_2(Q_0)$, $\sigma_4 \sigma_2(Q_0)$, and $\sigma_5 \sigma_4 \sigma_2(Q_0)$. 
% These tableaux are the following.
\begin{center}
$\sigma_2(Q_0) =\begin{ytableau}
1 & 2 & 6 & 7 & \dots & n \\
3 & 5 \\
4
\end{ytableau}$, 
\qquad 
$\sigma_4 \sigma_2(Q_0)=  \begin{ytableau}
1 & 2 & 6 & 7 & \dots & n \\
3 & 4 \\
5
\end{ytableau}$
\end{center}
\begin{center}
$\sigma_5 \sigma_4 \sigma_2(Q_0)= \begin{ytableau}
1 & 2 & 5 & 7 & \dots & n \\
3 & 4 \\
6
\end{ytableau}$
\end{center}
By performing the inverse RS algorithm, we see that the involutions 
whose RS tableaux are 
$\sigma_2(Q_0)$, $\sigma_4 \sigma_2(Q_0)$, and $\sigma_5 \sigma_4 \sigma_2(Q_0)$ are $(14)(25)$, $(13)(25)$, and $(13)(26)$, respectively. 
Their steady-state times 
are $1$, $2$, and $3$, respectively.

We now calculate the steady-state time for the rest of the tableaux in this sequence.
Fix $6 \leq k \leq n-1$, and 
let 
\[
Q_k \coloneqq \sigma_k \sigma_{k-1}...\sigma_6\sigma_5\sigma_4\sigma_2(Q_0)
=
\ytableausetup{boxsize=1.7em}
\begin{ytableau} 1 & 2 & 5 & 6 & ... & k & \scalebox{.7}{$k+2$} & ... & n \\ 3 & 4 \\ \scalebox{.7}{$k+1$} \end{ytableau}.
\]
Its corresponding involution is $W_k \coloneqq \RSK^{-1}(Q_k,Q_k)=(13)(2,k+1)$. 
We will show that $W_k$ has steady-state time $k-2$. The configuration at time $t=0$ is 
the one-line notation of 
$W_k$:
\begin{equation*}
     3 \, (k+1) \, 1 \, 4 \, 5 \, 6 \hdots \hdots k \, 2 \, (k+2) \, \hdots \, n
\end{equation*}
At $t=1$ we have the configuration
\begin{equation*}
\BB^1(W_k)=
e \, 
e \,
\underbrace{3 \, (k+1)}_{\text{increasing run}} 
    \overbrace{
    e \hdots e}^{k-4  \text{ copies}} 
{4}
\overbrace{e \hdots e}^{  \substack{  n-k-1\\ \text{copies}}} \underbrace{1256 \hdots k \, (k+2) \hdots n}_{\text{increasing run}}
\end{equation*}
which is of the form given in~\eqref{eq:thm:max SST of T with sh(n-3,2,1):BB1:lemma} in Lemma~\ref{lemma:thm:max SST of T with sh(n-3,2,1)}.
Therefore, $\SST(W_k)=1+(k-4)+1=k-2$ and $W_k$ is 
good.  
Indeed, we have
\begin{align*}
\BB^{k-4}(W_k)=
3 (k+1) e 4 e \hdots e & 1256 \hdots k (k+2) \hdots n\\
\BB^{k-3}(W_k) = 3 (k+1) 4 e \hdots ee & 1256 \hdots k (k+2) \hdots n \\
\label{eq:SS of k-2}
\BB^{k-2}(W_k) = (k+1) 3 4 e \hdots ee &1256 \hdots k (k+2) \hdots n
\end{align*}
so $\BB^{k-2}(W_k)$ is in steady state, but $\BB^{k-3}(W_k)$ is not; 
in addition, $\SD(W_k)=\P(W_k)=Q_k$, so $Q_k$ is good.
\end{proof}

\begin{example}
\label{ex:425136}
Consider $\omega = 452361$. Using Proposition~\ref{prop:chain of Q with SST 0 to n-3}, we can create a sequence of tableaux that have steady-state times $0$, $1$, $2$, and $3$:
\begin{equation*}
Q_0 =\young(136,25,4), 
\sigma_2(Q_0) = 
\young(126,35,4),
\sigma_4\sigma_2(Q_0) = 
\young(126,34,5),
\sigma_5 \sigma_4 \sigma_2 (Q_0) = 
\young(125,34,6)
\end{equation*}
The corresponding involutions are 
\begin{align*}
(14)(35), \quad (14)(25), \quad (13)(25), \quad &\text{ and } \quad (13)(26) \text{ in cycle notation, and} \\
425136, \qquad 453126, \qquad 351426, \quad &\text{ and }  \quad 361452 \text{ in one-line notation},
\end{align*}
in this order.
\end{example}

%%%%
%%%%
\section{Further directions}
\label{sec:further directions}
Recall that a permutation $\omega$ is 
 \emph{(BBS) good} %or \emph{good} 
if 
$\SD(\omega)$ is standard (equivalently, 
$\SD(\omega) = \P(\omega)$, due to Theorem~\ref{thm:tfae}). 
If a permutation is not good, let us call it \emph{bad}.

\subsection{Classical permutation patterns}

A permutation, or \emph{pattern}, $\sigma$ is said to be \emph{contained in}, or to be a \emph{subpermutation} of, another
permutation $\omega$ if $\omega$ has a (not necessarily contiguous) subsequence
whose elements are 
in the same relative order as $\sigma$, 
alternatively, $\omega$ has a subsequence whose standardization is equal to $\sigma$. 
If $\omega$ does not contain $\sigma$, we say that $\omega$ avoids $\sigma$. 
For example, $31\mathbf{4}5\mathbf{9}2\mathbf{6}\mathbf{8}7$ contains
1423 because the subsequence 4968 (among others) is ordered in the same way as 1423.
On the other hand, $314592687$ avoids $3241$ since $314592687$
has no subsequence ordered in the same way as $3241$. For more details, see for example the note~\cite{Bev15}.

The above notion of pattern containment and pattern avoidance is sometimes referred to as \emph{classical}. 
It turns out that classical pattern avoidance is too restrictive to be used to find all good permutations. 
The following shows that there are good permutations which contain bad patterns.

\begin{example}
\label{ex:good permutations are not closed under pattern containment}
% Starting with $n=5$, a good permutation in $S_n$ may have a (non-consecutive) subpattern which is not good.
A good permutation may have a subpattern which is not good.
\begin{enumerate}[a.)]
\item 
The permutation 
$2 5 1 4 3$ is good, % checked with Sage % not an involution
but it has a subpermutation $2 1 4 3$ which is bad.

\item 
The permutation $3 5 1 4 2$ % an involution
is good, % verified with Sage
but its subpermutation $3 1 4 2$ is bad. % verified with Sage

\item 
Let $\omega= 4 2 5 1 3$, % an involution
which is a good permutation, % verified with Sage
and let $\sigma=4253$, a subsequence of $\omega$. 
The standardization of $\sigma$ is 
$3 1 4 2$, which is a bad permutation.
\end{enumerate}
\end{example}

\begin{remark}
Example~\ref{ex:good permutations are not closed under pattern containment} shows that the good permutations are \emph{not} closed under classical pattern containment.
This means that the set of 
good permutations \emph{cannot} be characterized by a set of classically avoided patterns.
\end{remark}

Although it is impossible to characterize good permutations using classical pattern avoidance, we can give an instance where classical pattern avoidance can be used to find a (proper) subset of good permutations.
The following is straightforward to prove using a localized version of Greene's theorem (see \cite[Section 2.2]{DGGRS23}) and Theorem~\ref{thm:tfae}.
\begin{proposition}\label{prop:if w avoids classical pattern 2143 and 3142 then w is good}
If %a permutation 
$\omega$ avoids both the classical pattern
$2143$
and $3142$, 
then $\omega$ is good. 
\end{proposition}

\begin{remark}
The converse of Proposition~\ref{prop:if w avoids classical pattern 2143 and 3142 then w is good} is false. 
As shown in Example~\ref{ex:good permutations are not closed under pattern containment},
there are good permutations which have the classical pattern $2143$ or $3142$. 
\end{remark}

\subsection{Consecutive permutation patterns}

A permutation, or pattern, $\sigma$ is said to be a  \emph{consecutive pattern} of another
permutation $\omega$ if $\omega$ has a consecutive subsequence whose elements are in the same relative order as $\sigma$. 
Otherwise, $\omega$ is said to avoid $\sigma$ as a consecutive pattern.
For example, 314\textbf{5926}87 contains
2413 because the subsequence 5926
 is ordered in the same way as 2413.
On the other hand, $314592687$ avoids $321$ since $314592687$
has no consecutive subsequence ordered in the same way as $321$ (although $314592687$ contains the classical pattern $321$). 
Consecutive pattern was first systematically studied by Elizalde and Noy in the early 2000s~\cite{ElizaldeNoy03}. Since then, this notion has appeared in dynamical systems~\cite{ElizaldeSIAM09}, generalized to arbitrary Coxeter groups~\cite{GaoWang23}, and more; see the survey~\cite{ElizaldeSurvey}.

We conjecture that good permutations are closed under consecutive pattern containment; that is, if a permutation is good, then any consecutive subpermutation is also good.

\begin{conjecture}
\label{conj:good permutations closed}
If a permutation $\omega$ is good, 
then the standardization of every consecutive subpattern of $\omega$ is also good.
\end{conjecture}

\subsection{Motzkin numbers}

The $n$th Motzkin number is the number of ways to draw nonintersecting chords between $n$ labeled points on a circle~\cite{Motzkin48}. 
They also count the number of Motzkin paths,  
$4321$-avoiding involutions, 
% standard tableaux having at most three rows, 
along with many other 
objects. See~\cite[A001006]{oeis} % https://oeis.org/A001006
and~\cite{DonagheyShapiro77}.

\begin{conjecture}\label{conj:Motzkin}
The number of size-$n$ good tableaux is equal to the $n$th Motzkin number.
\end{conjecture}

\begin{remark}
Since drawing nonintersecting chords between labeled points on a circle is equivalent to determining a noncrossing involution, we get that the number of noncrossing involutions in $S_n$ is equal to the $n$th Motzkin number. 
However, 
Proposition~\ref{prop:noncrossing w is good iff w is nested} shows 
that   
some noncrossing involutions are good and some noncrossing involutions are not, 
so the set of good involutions is not equal to the set of noncrossing involutions. 

It is also known that the number of nonnesting involutions in $S_n$ is equal to the $n$th Motzkin number.   
Proposition~\ref{prop:noncrossing w is good iff w is nested}
illustrates that the set of good involutions is not equal to the set of nonnesting involutions.
\end{remark}

%%%
%%%
\section*{Acknowledgments}
This work was supported by the Summer Undergraduate Math Research at Yale (SUMRY 2021) program and the NSF (REU Site grant DMS-2050398). 
We thank Sergi Elizalde and Joel Lewis for helpful discussions, 
and 
Su Ji Hong,
Matthew Li,
Raina Okonogi-Neth, 
Mykola Sapronov, 
Dash Stevanovich, and
Hailey Weingord
for useful input and feedback during SUMRY 2022. 
We are grateful to Darij Grinberg and to the anonymous reviewer whose suggestions helped improve and clarify this paper. 
This research also benefited from the open-source software {\sc SageMath}~\cite{sagemath}.

 % \bibliographystyle{alpha}
 % \bibliography{bibbib} 

\end{document}